\documentclass[11pt]{amsart}

\usepackage[usenames,dvipsnames,svgnames,table]{xcolor}
\usepackage[colorlinks=true, pdfstartview=FitV, linkcolor=blue, citecolor=blue, urlcolor=blue]{hyperref}

\usepackage[centering]{geometry}                
\usepackage{graphicx}
\usepackage{amssymb}
\usepackage{epstopdf}
\usepackage{lscape}
\usepackage{xcolor}
\usepackage{tabmac}
\usepackage[utf8]{inputenc}
\usepackage{tikz,caption}
\usepackage{enumitem}
\setlist[itemize]{leftmargin=2em}
\setlist[enumerate]{leftmargin=2em}
\usepackage{booktabs}
\usepackage{multirow}
\usepackage{mathtools}
\usepackage{mathrsfs}
\usepackage[aligntableaux=center, smalltableaux]{ytableau}
\usepackage{youngtab}

\definecolor{darkblue}{rgb}{0.0,0,0.7} 
\definecolor{darkred}{rgb}{0.7,0,0} 
\definecolor{darkgreen}{rgb}{0, .6, 0} 

\newcommand{\defncolor}{\color{darkred}}
\newcommand{\defn}[1]{{\defncolor\emph{#1}}} 

\newtheorem{theorem}{Theorem}[section]
\newtheorem{proposition}[theorem]{Proposition}
\newtheorem{corollary}[theorem]{Corollary}

\newtheorem{conjecture}[theorem]{Conjecture}
\theoremstyle{definition}
\newtheorem{definition}[theorem]{Definition}
\newtheorem{example}[theorem]{Example}
\newtheorem{remark}[theorem]{Remark}
\numberwithin{equation}{section}

\newcommand{\idiot}[1]{\vspace{5 mm}\par \noindent
\marginpar{\textsc{Note}}
\framebox{\begin{minipage}[c]{0.95 \textwidth}
#1 \end{minipage}}\vspace{5 mm}\par}

\renewcommand{\idiot}[1]{}

\newcommand{\rank}{\mathsf{rank}}

\def\NN{{\mathbb N}}

\def\ZZ{{\mathbb Z}}

\usepackage[colorinlistoftodos]{todonotes}

\newdimen\squaresize \squaresize=10pt
\newdimen\thickness \thickness=0.4pt

\def\square#1{\hbox{\vrule width \thickness
     \vbox to \squaresize{\hrule height \thickness\vss
        \hbox to \squaresize{\hss#1\hss}
     \vss\hrule height\thickness}
\unskip\vrule width \thickness}
\kern-\thickness}

\def\vsquare#1{\vbox{\square{$#1$}}\kern-\thickness}

\def\thisbox#1{\kern-.09ex\fbox{#1}}
\def\downbox#1{\lower1.200em\hbox{#1}}
\newdimen\Squaresize \Squaresize=20pt
\newdimen\Thickness \Thickness=0.4pt

\def\Square#1{\hbox{\vrule width \Thickness
     \vbox to \Squaresize{\hrule height \Thickness\vss
        \hbox to \Squaresize{\hss#1\hss}
     \vss\hrule height\Thickness}
\unskip\vrule width \Thickness}
\kern-\Thickness}

\def\Vsquare#1{\vbox{\Square{$#1$}}\kern-\Thickness}

\title[From quasi-symmetric to Schur expansions]{From quasi-symmetric to Schur expansions \\ with applications to \\
symmetric chain decompositions and plethysm}

\author[Orellana]{Rosa Orellana}
\address[R. Orellana]{Mathematics Department, Dartmouth College,
Hanover, NH 03755, U.S.A.}
\email{Rosa.C.Orellana@dartmouth.edu}
\urladdr{\href{https://math.dartmouth.edu/~orellana/}{https://math.dartmouth.edu/~orellana/}}

\author[Saliola]{Franco Saliola}
\address[F. Saliola]{LACIM, D\'epartement de math\'ematiques,
Universit\'e du Qu\'ebec \`a Montr\'eal, Canada}
\email{saliola.franco@uqam.ca}

\author[Schilling]{Anne Schilling}
\address[A. Schilling]{Department of Mathematics, University of California, One Shields
Avenue, Davis, CA 95616-8633, U.S.A.}
\email{anne@math.ucdavis.edu}
\urladdr{\href{http://www.math.ucdavis.edu/~anne}{http://www.math.ucdavis.edu/~anne}}

\author[Zabrocki]{Mike Zabrocki}
\address[M. Zabrocki]{Department of Mathematics and Statistics,  York University, 4700 Keele Street, Toronto,
Ontario M3J 1P3, Canada}
\email{zabrocki@mathstat.yorku.ca}
\urladdr{\href{http://garsia.math.yorku.ca/~zabrocki/}{http://garsia.math.yorku.ca/~zabrocki/}}


\begin{document}

\begin{abstract}
It is an important problem in algebraic combinatorics to deduce the Schur function expansion of a symmetric function whose
expansion in terms of the fundamental quasisymmetric function is known. For example,
formulas are known for the fundamental expansion of a Macdonald symmetric function and for the plethysm of two
Schur functions, while the Schur expansions of these expressions are still elusive.
Based on work of Egge, Loehr and Warrington, Garsia and Remmel provided a method to obtain the Schur expansion
from the fundamental expansion by replacing each quasisymmetric function by a Schur function
(not necessarily indexed by a partition) and using straightening rules to obtain the Schur
expansion. Here we provide a new method that only involves the coefficients of the quasisymmetric functions
indexed by partitions and the quasi-Kostka matrix. As an application, we identify the lexicographically largest term
in the Schur expansion of the plethysm of two Schur functions. We provide the Schur expansion of
$s_w[s_h](x,y)$ for $w=2,3,4$ using novel symmetric chain decompositions of Young's lattice for partitions in a $w\times h$ box.
For $w=4$, this is the first known combinatorial expression for the coefficient of
$s_{\lambda}$ in $s_{w}[s_{h}]$ for two-row partitions $\lambda$, and for $w=3$ the combinatorial expression is new.
\end{abstract}

\maketitle

\section{Introduction}

In 1984 Gessel \cite{Gessel.1984} introduced the fundamental quasisymmetric functions
$F_\alpha$ and showed that the coefficient of the quasisymmetric
$F_\alpha$ in a Schur function $s_\lambda$ is the number of
standard tableaux of shape $\lambda$ with descent composition equal to $\alpha$.
In 2010, Egge, Loehr and Warrington~\cite{ELW.2010} gave a combinatorial formula for the
right inverse of the transition matrix between the Schur basis of the symmetric functions
and the fundamental quasisymmetric function basis.
This formula provides a method for transforming the
expansion of a symmetric function of homogeneous degree $n$
in terms of Gessel's fundamental quasisymmetric functions into an expansion in terms of Schur functions.

Garsia and Remmel~\cite{GR.2018} reformulated this as follows.
Let $F_\alpha$ denote the fundamental quasisymmetric function indexed by the composition $\alpha$.
If a symmetric function $f$ expands as
\begin{equation}
\label{equation.quasi expansion}
    f = \sum_{\alpha \models n} c_\alpha F_\alpha
\end{equation}
in the basis of quasisymmetric functions, then
\[
    f = \sum_{\alpha \models n} c_\alpha s_\alpha,
\]
where $s_\alpha$ are Schur functions indexed by compositions. To obtain the expansion in terms of Schur functions $s_\lambda$ indexed by
partitions, one needs to use a straightening rule for Schur functions. More recently, Gessel~\cite{Gessel.2019} gave a sign-reversing involution proof
of this formula.

There are many important examples in symmetric function theory,
algebraic combinatorics and representation theory, where a quasisymmetric
expansion of a symmetric function is known, but the Schur expansion remains
elusive. For example, combinatorial expressions for the quasisymmetric expansion of
LLT polynomials, modified Macdonald polynomials~\cite{HHL.2005},
characters of higher Lie modules (or Thrall's problem)~\cite{GR.1993}
or the plethysm of two Schur functions~\cite{LW.2012} exist and it is
desirable to deduce the Schur expansions from these quasisymmetric expansions.

Here we present a novel method to transition from a fundamental quasisymmetric expansion
to a Schur expansion that requires only the coefficients
$c_\lambda$ in~\eqref{equation.quasi expansion} indexed by partitions $\lambda$.
We do this by giving a right inverse of the transition matrix between the
Schur and fundamental quasisymmetric function basis that is different from the
one given by Egge, Loehr and Warrington~\cite{ELW.2010}. Our right inverse
has the property that only the rows indexed by partitions have non-zero entries.
We present our interpretation in terms of the quasi-Kostka matrix that counts
the number of quasi-Yamanouchi tableaux~\cite{AS.2017,Wang.2019}.
This interpretation is equivalent to, but different from, the
fundamental expansion of a Schur function
provided by Gessel~\cite{Gessel.1984}.

As an application of our methods, we identify the lexicographically largest term that occurs in the plethysm of two
Schur functions starting from the quasisymmetric expansion provided by Loehr and Warrington~\cite{LW.2012}.
This question was posed to us by Panova and Zhao~\cite{PZ.2023} in the quest to determine whether
the plethysm $s_\lambda[s_\mu]$ of two Schur functions has a saturated Newton polytope.
Our method gives a simple, alternative proof of~\cite[Corollary 9.1]{PW.2019} and~\cite[Theorem 4.2]{I.2011}.
In addition, we compute the Schur expansion of $s_w[s_h](x,y)$ for $w=2,3,4$
using novel symmetric chain decompositions of Young's lattice for partitions in a $w\times h$ box.
Our symmetric chain decompositions have properties (that we call the \emph{extension}, \emph{restriction} and \emph{pattern} properties) that existing
symmetric chain decompositions~\cite{Lindstrom1980,OHara1990,Riess1978,Wen2004,West1980} do not have.
Our explicit combinatorial expressions for the Schur coefficients for $w=3$ and $w=4$ are new.

The paper is organized as follows. In Section~\ref{section.F}, we review Gessel's quasisymmetric functions and give the
expansion of Schur functions in terms of quasisymmetric functions using standard tableaux and quasi-Yamanouchi tableaux.
In Section~\ref{section.main}, we show that the Schur expansion of a symmetric function can be deduced from the quasisymmetric
expansion knowing only the coefficients of terms indexed by partitions using the quasi-Kostka matrix (see Theorem~\ref{theorem.main}).
We study the inverse of the quasi-Kostka matrix in Section~\ref{section.inverse} in analogy to results by Egge, Loehr and Warrington~\cite{ELW.2010}
for the inverse Kostka matrix. (It is important to note that Egge, Loehr and Warrington invert a rectangular matrix, whereas the
quasi-Kostka matrix is a square matrix.) We conclude in Section~\ref{section.applications} with applications.
In Section~\ref{section.leading term} we determine the leading term of the plethysm of two Schur functions and
in Section~\ref{section.two variables} we compute the Schur expansion of $s_w[s_h](x,y)$ for $w=2,3,4$
(see Corollaries~\ref{cor.b w=3} and~\ref{cor.b w=4}) and give the novel symmetric chain decompositions.

\section*{Acknowledgement}
We thank Greta Panova and Chenchen Zhao for discussions and sharing preprint~\cite{PZ.2023} with us.
We thank Laura Colmenarejo for initial collaboration on symmetric chain decompositions of $L(w,h)$.
We thank Alvaro Gutierrez Caceres for private communication in March 2024 on $\mathfrak{sl}_2$-crystal structures for $\mathsf{Alt}^w \mathsf{Sym}^h V$
for $w = 2, 3, 4$, and $\mathsf{dim} V = 2$. This is related to the symmetric chain decompositions on $L(w,h)$ for $w=2,3,4$
in Section~\ref{section.two variables}.

AS thanks IPAM for hospitality at the program ``Geometry, Statistical Mechanics, and Integrability'', where this work was completed.

This work was partially supported by NSF grants DMS--2153998, DMS--2053350 and NSERC/CNSRG.

\section{Quasisymmetric functions}
\label{section.F}

Gessel~\cite{Gessel.1984} introduced the \defn{fundamental quasisymmetric functions} $F_\alpha$ indexed by compositions $\alpha$ as
\begin{equation}
	F_\alpha = \sum_{\beta \preccurlyeq \alpha} M_\beta,
\end{equation}
where $\beta \preccurlyeq \alpha$ denotes that composition $\beta$ is a refinement of composition $\alpha$, and
\[
	M_\beta = \sum_{i_1<i_2<\cdots<i_\ell} x_{i_1}^{\beta_1} x_{i_2}^{\beta_2} \cdots x_{i_\ell}^{\beta_\ell}.
\]

For a standard tableau $T$, the letter $i$ is a \defn{descent} if the letter $i+1$ is in a higher row of the tableau (in French notation). Denote the
descents of $T$ by $d_1<d_2<\cdots<d_k$. The \defn{descent composition} is defined as
\[
	\mathsf{desComp}(T) = ( d_1, d_2 - d_1, \ldots, d_k - d_{k-1}, n-d_k),
\]
where $n$ is the size of the shape of $T$.

The expansion of a Schur function in terms of the fundamental quasisymmetric functions is given by~\cite{Gessel.1984,Gessel.2019}
\begin{equation}
\label{equation.Schur quasi}
	s_\lambda = \sum_{T \in \mathsf{SYT}(\lambda)} F_{\mathsf{desComp}(T)},
\end{equation}
where $\mathsf{SYT}(\lambda)$ is the set of standard tableaux of shape $\lambda$.

\begin{definition}[\cite{AS.2017,Wang.2019}]
A semistandard Young tableau $T$ is a \defn{quasi-Yamanouchi tableau} if when $i>1$ appears in the tableau, some instance
of $i$ is in a higher row than some instance of $i - 1$ for all $i$. Denote the set of all quasi-Yamanouchi tableaux of shape $\lambda$
by $\mathsf{QYT}(\lambda)$ and by $\mathsf{QYT}_{\leqslant m}(\lambda)$ the subset of $\mathsf{QYT}(\lambda)$ with largest entry at most $m$.
\end{definition}

\begin{example}
\label{example.quasiYamanouchi}
The only quasi-Yamanouchi tableaux of shape $(4,2,1)$ and weight $(2,2,2,1)$ are
\[
\raisebox{1.2cm}{
{\def\lr#1{\multicolumn{1}{|@{\hspace{.6ex}}c@{\hspace{.6ex}}|}{\raisebox{-.3ex}{$#1$}}}
\raisebox{-.6ex}{$\begin{array}[t]{*{4}c}\cline{1-1}
\lr{4}\\\cline{1-2}
\lr{2}&\lr{3}\\\cline{1-4}
\lr{1}&\lr{1}&\lr{2}&\lr{3}\\\cline{1-4}
\end{array}$}
}}
\qquad \text{and} \qquad
\raisebox{1.2cm}{
{\def\lr#1{\multicolumn{1}{|@{\hspace{.6ex}}c@{\hspace{.6ex}}|}{\raisebox{-.3ex}{$#1$}}}
\raisebox{-.6ex}{$\begin{array}[t]{*{4}c}\cline{1-1}
\lr{3}\\\cline{1-2}
\lr{2}&\lr{4}\\\cline{1-4}
\lr{1}&\lr{1}&\lr{2}&\lr{3}\\\cline{1-4}
\end{array}$}
}}.
\]
On the other hand
\[
\raisebox{1cm}{
{\def\lr#1{\multicolumn{1}{|@{\hspace{.6ex}}c@{\hspace{.6ex}}|}{\raisebox{-.3ex}{$#1$}}}
\raisebox{-.6ex}{$\begin{array}[t]{*{4}c}\cline{1-1}
\lr{4}\\\cline{1-2}
\lr{3}&\lr{3}\\\cline{1-4}
\lr{1}&\lr{1}&\lr{2}&\lr{2}\\\cline{1-4}
\end{array}$}
}}
\]
is not a quasi-Yamanouchi tableau since there is no $i=2$ in a higher row than some instance of 1.
\end{example}

The expansion in~\eqref{equation.Schur quasi} can be reformulated in terms of quasi-Yamanouchi tableaux via the standardization map.
Let $T$ be a semistandard Young tableau of weight $\alpha=(\alpha_1,\ldots,\alpha_\ell)$. Then the \defn{standardization} $\mathsf{standard}(T)$
is obtained from $T$ by replacing the letters $i$ in $T$ from left to right by $\alpha_1+\alpha_2+\cdots+\alpha_{i-1}+1,\ldots,
\alpha_1+\alpha_2+\cdots+\alpha_i$, which is a standard tableau of the same shape. Conversely, for a standard tableau $T$ with descents
$d_1<d_2<\cdots<d_k$, the \defn{de-standardization} $\mathsf{destandard}(T)$
is obtained from $T$ by replacing the letters $d_{i-1}+1,d_{i-1}+2,\ldots,d_i$
by $i$, where by convention $d_0=0$ and $d_{k+1}$ is the size of the shape of $T$. Note that
\[
	\mathsf{destandard} \colon \mathsf{SYT}(\lambda) \to \mathsf{QYT}(\lambda)
\]
is a bijection and its inverse is the map $\mathsf{standard}$. Hence we have
\begin{equation}
\label{equation.QYT expansion}
	s_\lambda = \sum_{T \in \mathsf{QYT}(\lambda)} F_{\mathsf{wt}(T)}
\end{equation}
where $\mathsf{wt}(T)$ represents the weight of the tableau $T$.
When restricting to Schur polynomials in $m$ variables this becomes~\cite[Theorem 2.7]{AS.2017}
\begin{equation}
\label{equation.finite variables}
	s_\lambda(x_1,\ldots,x_m) = \sum_{T \in \mathsf{QYT}_{\leqslant m}(\lambda)} F_{\mathsf{wt}(T)}(x_1,\ldots,x_m).
\end{equation}

\section{From quasisymmetric to Schur expansions}
\label{section.main}

In this section, we describe a novel way to obtain the Schur expansion of a function $f$ as in~\eqref{equation.quasi expansion}, whose expansion
in terms of the fundamental quasisymmetric functions is known. We begin by defining the quasi-Kostka matrix.  Note that we order the partitions of
$n$ in reverse lexicographic order.

\begin{definition}
Fix a positive integer $n$.
Let $\mathsf{QK}_{\lambda,\alpha}$ be the number of tableaux in $\mathsf{QYT}(\lambda)$ of weight $\alpha$, where $\lambda \vdash n$ is
a partition of $n$ and $\alpha \models n$ is a composition of $n$.
Define the \defn{quasi-Kostka matrix} as
\[
	\mathsf{Q} = (\mathsf{QK}_{\lambda,\mu})_{\lambda,\mu},
\]
where $\lambda$ and $\mu$ are partitions of $n$.
\end{definition}

\begin{example}
From Example~\ref{example.quasiYamanouchi}, we have $\mathsf{QK}_{(4,2,1),(2,2,2,1)}=2$.
\end{example}

\begin{example}
\label{example.QK7}
Order the partitions of $n=7$ in reverse lexicographic order (where we use the frequency notation $a^m$ if part $a$ appears $m$ times)
\[
 	(7), (61), (52), (51^2), (43), (421), (41^3),(3^2 1),(32^2),(321^2),(31^4), (2^31), (2^21^3), (21^5), (1^7).
\]
Then (where $\mathsf{QK}_{(4,2,1),(2,2,2,1)}=2$ is highlighted in red)
\[
\mathsf{Q} =
\left(\begin{array}{rrrrrrrrrrrrrrr}
1 & 0 & 0 & 0 & 0 & 0 & 0 & 0 & 0 & 0 & 0 & 0 & 0 & 0 & 0 \\
0 & 1 & 1 & 0 & 1 & 0 & 0 & 0 & 0 & 0 & 0 & 0 & 0 & 0 & 0 \\
0 & 0 & 1 & 0 & 1 & 1 & 0 & 1 & 1 & 0 & 0 & 0 & 0 & 0 & 0 \\
0 & 0 & 0 & 1 & 0 & 1 & 0 & 1 & 1 & 0 & 0 & 0 & 0 & 0 & 0 \\
0 & 0 & 0 & 0 & 1 & 0 & 0 & 1 & 1 & 0 & 0 & 1 & 0 & 0 & 0 \\
0 & 0 & 0 & 0 & 0 & 1 & 0 & 1 & 2 & 1 & 0 & \textcolor{darkred}{2} & 0 & 0 & 0 \\
0 & 0 & 0 & 0 & 0 & 0 & 1 & 0 & 0 & 1 & 0 & 1 & 0 & 0 & 0 \\
0 & 0 & 0 & 0 & 0 & 0 & 0 & 1 & 1 & 0 & 0 & 2 & 0 & 0 & 0 \\
0 & 0 & 0 & 0 & 0 & 0 & 0 & 0 & 1 & 0 & 0 & 2 & 0 & 0 & 0 \\
0 & 0 & 0 & 0 & 0 & 0 & 0 & 0 & 0 & 1 & 0 & 2 & 1 & 0 & 0 \\
0 & 0 & 0 & 0 & 0 & 0 & 0 & 0 & 0 & 0 & 1 & 0 & 1 & 0 & 0 \\
0 & 0 & 0 & 0 & 0 & 0 & 0 & 0 & 0 & 0 & 0 & 1 & 0 & 0 & 0 \\
0 & 0 & 0 & 0 & 0 & 0 & 0 & 0 & 0 & 0 & 0 & 0 & 1 & 0 & 0 \\
0 & 0 & 0 & 0 & 0 & 0 & 0 & 0 & 0 & 0 & 0 & 0 & 0 & 1 & 0 \\
0 & 0 & 0 & 0 & 0 & 0 & 0 & 0 & 0 & 0 & 0 & 0 & 0 & 0 & 1
\end{array}\right).
\]
\end{example}

Equation~\eqref{equation.QYT expansion} can be rewritten as
\begin{equation}
\label{equation.s QK F}
	s_\lambda = \sum_{\alpha \models |\lambda|} \mathsf{QK}_{\lambda,\alpha} F_\alpha.
\end{equation}

\begin{proposition}
\label{proposition.prop QK}
For a partition $\lambda \vdash n$ and a composition $\alpha \models n$, we have
\begin{enumerate}
\item $\mathsf{QK}_{\lambda,\lambda}=1$;
\item $\mathsf{QK}_{\lambda,\alpha}=0$ if $\lambda<\alpha$ in dominance or lexicographic order;
\item The quasi-Kostka matrix $\mathsf{Q}$ is unit upper-triangular (if rows and columns are indexed by partitions of $n$ in reverse lexicographic order).
\end{enumerate}
\end{proposition}

\begin{proof}
Let $K_{\lambda,\alpha} = |\mathsf{SSYT}(\lambda,\alpha)|$ be the Kostka number, which is the number of semistandard Young tableaux of shape
$\lambda$ and weight $\alpha$. Since quasi-Yamanouchi tableaux are semistandard tableaux with extra conditions, we have
$0 \leqslant \mathsf{QK}_{\lambda,\alpha} \leqslant K_{\lambda,\alpha}$. It is well-known (see for example~\cite{Stanley}) that $K_{\lambda,\alpha}=0$
for $\lambda<\alpha$ in dominance and lexicographic order. This proves part 2.

The only semistandard tableau in $\mathsf{SSYT}(\lambda,\lambda)$ is the tableau with $i$'s in row $i$. This tableau is also quasi-Yamanouchi,
which implies part 1. Part 3 is a direct consequence of parts 1 and 2.
\end{proof}

\begin{corollary}
\label{cor.lex}
For any partition $\lambda$,
\[
	s_\lambda = F_\lambda + \text{terms that are strictly smaller than $\lambda$ in lexicographic order}.
\]
\end{corollary}

\begin{corollary}
\label{cor.b=c}
If
\begin{equation}
\label{equation.s-F}
	f = \sum_{\lambda \vdash n} b_\lambda s_\lambda = \sum_{\alpha \models n} c_\alpha F_\alpha
\end{equation}
is a symmetric function, let $\mathsf{supp}(f) = \{ \alpha \models n \mid c_\alpha \neq 0 \}$.
Let $\nu$ be the maximum element in $\mathsf{supp}(f)$ with respect to the lexicographic order
on compositions. Then $\nu \vdash n$ and $b_\nu = c_\nu$.
\end{corollary}

Our main result states that we can obtain the coefficients $b_\lambda$ in~\eqref{equation.s-F} from the coefficients $c_\mu$ in the quasisymmetric
expansion for partitions $\mu$ (instead of needing the coefficients $c_\alpha$ for all compositions $\alpha$).
A similar observation was made in~\cite[Section 6.2]{MG.2023}.

\begin{theorem}
\label{theorem.main}
For a symmetric function $f$ as in~\eqref{equation.s-F}, we have
\begin{equation}
\label{equation.b in terms of c}
	b_\lambda = \sum_{\mu \vdash n} \mathsf{Q}^{-1}_{\mu,\lambda} c_\mu.
\end{equation}
\end{theorem}

\begin{proof}
By Corollary~\ref{cor.b=c}, if $\nu$ is the lexicographically largest element in $\mathsf{supp}(f)$, then $\nu$ is a partition and $b_\nu = c_\nu$.
Consider $f-b_\nu s_\nu$, which is a symmetric function. The maximal element in $\mathsf{supp}(f-b_\nu s_\nu)$
is strictly smaller than $\nu$ in lexicographic order. Repeat this argument. Since the maximal element in lexicographic order in the support strictly
decreases and since $\{ F_\alpha\}$ is a basis for quasisymmetric functions, this shows that $b_\lambda$ only depends on $c_\mu$ for partitions $\mu$.

Since by~\eqref{equation.s QK F}
\[
	f = \sum_{\lambda \vdash n} b_\lambda s_\lambda
	= \sum_{\lambda \vdash n} b_\lambda \Bigl( \sum_{\mu \vdash n} \mathsf{QK}_{\lambda,\mu} F_\mu +
	\sum_{\substack{\alpha \models n\\ \alpha \text{ not partition}}} \mathsf{QK}_{\lambda,\alpha} F_\alpha \Bigr)
\]
we have that $c_\mu = \sum_{\lambda \vdash n} b_\lambda \mathsf{QK}_{\lambda,\mu}$ or, equivalently,
Equation~\eqref{equation.b in terms of c} is the inverse of this relationship.
\end{proof}

\section{Inverse quasi-Kostka matrix}
\label{section.inverse}

Given the importance of Theorem~\ref{theorem.main}, we study the properties
of the inverse of the quasi-Kostka matrix $\mathsf{Q}^{-1}$ in
this section.

\begin{proposition}
\label{proposition.Qinverse}
For partitions $\lambda, \mu \vdash n$, we have
\begin{enumerate}
\item $\mathsf{Q}^{-1}_{\lambda,\lambda}=1$;
\item $\mathsf{Q}^{-1}_{\mu,\lambda}=0$ if $\lambda>\mu$ in dominance or lexicographic order;
\item if $\lambda < \mu$,
\begin{equation}
\label{equation.Qinverse}
	\mathsf{Q}^{-1}_{\mu,\lambda} = \sum_{k \geqslant 1} (-1)^k \sum_{\mu > \mu^1 > \mu^2 > \cdots > \mu^k=\lambda}
	\mathsf{QK}_{\mu,\mu^1} \mathsf{QK}_{\mu^1,\mu^2} \cdots \mathsf{QK}_{\mu^{k-1},\lambda}.
\end{equation}
\end{enumerate}
\end{proposition}

\begin{proof}
Parts 1 and 2 follow from the fact that $\mathsf{Q}$ is unit upper-triangular by Proposition~\ref{proposition.prop QK}.
Part 3 follows because if $\mathsf{Q} = \mathsf{I} + \mathsf{N}$, where $\mathsf{I}$ is the identity and $\mathsf{N}$
is nilpotent, then $\mathsf{Q}^{-1} = \mathsf{I} + \sum_{k \geqslant 1} (-1)^k \mathsf{N}^k$.
\end{proof}

\begin{example}
The inverse of $\mathsf{Q}$ of Example~\ref{example.QK7} for $n=7$ is
\[
\mathsf{Q}^{-1} =
\left(\begin{array}{rrrrrrrrrrrrrrr}
1 & 0 & 0 & 0 & 0 & 0 & 0 & 0 & 0 & 0 & 0 & 0 & 0 & 0 & 0 \\
0 & 1 & -1 & 0 & 0 & 1 & 0 & 0 & -1 & -1 & 0 & 2 & 1 & 0 & 0 \\
0 & 0 & 1 & 0 & -1 & -1 & 0 & 1 & 1 & 1 & 0 & -3 & -1 & 0 & 0 \\
0 & 0 & 0 & 1 & 0 & -1 & 0 & 0 & 1 & 1 & 0 & -2 & -1 & 0 & 0 \\
0 & 0 & 0 & 0 & 1 & 0 & 0 & -1 & 0 & 0 & 0 & 1 & 0 & 0 & 0 \\
0 & 0 & 0 & 0 & 0 & 1 & 0 & -1 & -1 & -1 & 0 & 4 & 1 & 0 & 0 \\
0 & 0 & 0 & 0 & 0 & 0 & 1 & 0 & 0 & -1 & 0 & 1 & 1 & 0 & 0 \\
0 & 0 & 0 & 0 & 0 & 0 & 0 & 1 & -1 & 0 & 0 & 0 & 0 & 0 & 0 \\
0 & 0 & 0 & 0 & 0 & 0 & 0 & 0 & 1 & 0 & 0 & -2 & 0 & 0 & 0 \\
0 & 0 & 0 & 0 & 0 & 0 & 0 & 0 & 0 & 1 & 0 & -2 & -1 & 0 & 0 \\
0 & 0 & 0 & 0 & 0 & 0 & 0 & 0 & 0 & 0 & 1 & 0 & -1 & 0 & 0 \\
0 & 0 & 0 & 0 & 0 & 0 & 0 & 0 & 0 & 0 & 0 & 1 & 0 & 0 & 0 \\
0 & 0 & 0 & 0 & 0 & 0 & 0 & 0 & 0 & 0 & 0 & 0 & 1 & 0 & 0 \\
0 & 0 & 0 & 0 & 0 & 0 & 0 & 0 & 0 & 0 & 0 & 0 & 0 & 1 & 0 \\
0 & 0 & 0 & 0 & 0 & 0 & 0 & 0 & 0 & 0 & 0 & 0 & 0 & 0 & 1
\end{array}\right).
\]
\end{example}

Combinatorially, we can interpret the expression for $\mathsf{Q}^{-1}_{\mu,\lambda}$
in~\eqref{equation.Qinverse} as signed chains (or sequences) of quasi-Yamanouchi tableaux,
as follows.

\begin{definition}
\label{definition.chains}
Define \defn{$\mathsf{Chains}(\mu, \lambda)$} to
be sequences $c = (T^{(1)}, T^{(2)}, \ldots, T^{(k+1)})$ of
quasi-Yamanouchi tableaux satisfying
\begin{enumerate}[topsep=1ex]
\item the shape of $T^{(1)}$ is $\mu$: that is, $T^{(1)} \in \mathsf{QYT}({\mu})$;
\item for $2 \leqslant i \leqslant k+1$, the shape of the $i$-th tableau is the weight of the $(i-1)$st tableau:
    that is, $T^{(i)} \in \mathsf{QYT}(\mathsf{wt}(T^{(i-1)}))$;
\item the shape and weight of the last tableau in the sequence is $\lambda$:
$T^{(k+1)} \in \mathsf{QYT}(\lambda)$ and $\mathsf{wt}(T^{(k+1)}) = \lambda$.
\end{enumerate}
The \defn{sign} associated to a chain $c$ is $(-1)^{\ell(c)-1}$,
where $\ell(c) = k+1$ is the length of the chain, and the \defn{weight $\mathsf{wt}(c)$} of $c$ is the weight of $T^{(k+1)}$.
\end{definition}

Since the number of sequences $(T^{(1)}, \ldots, T^{(k+1)})$ in $\mathsf{Chains}(\mu, \lambda)$
with $\mathsf{wt}(T^{(i)}) = \mu^{i}$ is
\[
    \mathsf{QK}_{\mu,\mu^1} \mathsf{QK}_{\mu^1,\mu^2} \cdots \mathsf{QK}_{\mu^{k-1},\lambda} \mathsf{QK}_{\lambda,\lambda},
\]
and $\mathsf{QK}_{\lambda,\lambda} = 1$,
the properties of Proposition~\ref{proposition.Qinverse} are equivalent to the statement
\[
	\mathsf{Q}^{-1}_{\mu,\lambda} = \sum_{c \in \mathsf{Chains}(\mu,\lambda)} (-1)^{\ell(c)-1}~.
\]

\begin{example}
\label{example.chains}
\allowdisplaybreaks
To illustrate Definition~\ref{definition.chains}, we have the following chains for $\mu=(4,1,1,1)$ and $\lambda=(2,2,2,1)$
\begin{align*}
    c_1 & =
    \left(~
        \begin{ytableau}
            {4}\\
            {3}\\
            {2}\\
            {1}&{1}&{2}&{3}\\
        \end{ytableau}
        , \quad
        \begin{ytableau}
            {4}\\
            {3}&{3}\\
            {2}&{2}\\
            {1}&{1}\\
        \end{ytableau}
    ~\right)
    & \mathsf{sign}(c_1)=-1
    \\
    c_2 & =
    \left(~
        \begin{ytableau}
            {4}\\
            {3}\\
            {2}\\
            {1}&{1}&{1}&{2}\\
        \end{ytableau}
        , \quad
        \begin{ytableau}
            {4}\\
            {3}\\
            {2}&{3}\\
            {1}&{1}&{2}\\
        \end{ytableau}
        , \quad
        \begin{ytableau}
            {4}\\
            {3}&{3}\\
            {2}&{2}\\
            {1}&{1}\\
        \end{ytableau}
    ~\right)
    & \mathsf{sign}(c_2)=1
    \\
    c_3 & =
    \left(~
        \begin{ytableau}
            {4}\\
            {3}\\
            {2}\\
            {1}&{1}&{1}&{2}\\
        \end{ytableau}
        , \quad
        \begin{ytableau}
            {4}\\
            {3}\\
            {2}&{2}\\
            {1}&{1}&{3}\\
        \end{ytableau}
        , \quad
        \begin{ytableau}
            {4}\\
            {3}&{3}\\
            {2}&{2}\\
            {1}&{1}\\
        \end{ytableau}
    ~\right)
    & \mathsf{sign}(c_3)=1
\end{align*}
resulting in $\mathsf{Q}^{-1}_{\mu,\lambda}=1$.
\end{example}

The next result shows that if we want to restrict to partitions whose length is bounded by $m$,
or to $m$ variables as in \eqref{equation.finite variables},
then we need only work with a submatrix of $\mathsf{Q}$.

\begin{corollary}
    \label{corollary-submatrix}
    Let $\mathsf{Q}_{\leqslant m}$ and $(\mathsf{Q}_{\leqslant m})^{-1}$
    denote the submatrix of the quasi-Kostka
    matrix $\mathsf{Q}$ and its inverse, respectively,
    corresponding to the rows and columns indexed by
    partitions of length at most $m$.
    Then
    \begin{equation*}
	    \left(\mathsf{Q}_{\leqslant m}\right)^{-1} = \left(\mathsf{Q}^{-1}\right)_{\leqslant m}.
    \end{equation*}
    In particular, if $\sum_{\lambda \vdash n} b_\lambda s_\lambda = \sum_{\alpha \models n} c_\alpha F_\alpha$
    and $b_\lambda = 0$ for $\ell(\lambda) > m$, then for $\ell(\lambda) \leqslant m$
    \begin{equation}
        \label{equation.b in terms of c in m variables}
        b_\lambda = \sum_{\substack{\mu \vdash n \\ \ell(\mu) \leqslant m}} \left(\mathsf{Q}_{\leqslant m}^{-1}\right)_{\mu,\lambda} c_\mu.
    \end{equation}
\end{corollary}

\begin{proof}
    This follows from the fact that the set of partitions of $n$ with length at
    most $m$ is an order ideal for reverse dominance order.
    Order the partitions by length (shortest to longest) and by reverse
    lexicographic order for partitions of the same length.
    This is a linear extension of reverse dominance order.
    If we order the rows and columns of $\mathsf{Q}$ according to this linear
    extension, we get an upper triangular matrix whose top left block is
    $\mathsf{Q}_{\leqslant m}$. Hence, the inverse of $\mathsf{Q}_{\leqslant
    m}$ is the top left block of the inverse of $\mathsf{Q}$.
\end{proof}

Let us reframe this result in a way similar to the Garsia--Remmel~\cite{GR.2018} formulation of the results of Egge, Loehr and
Warrington~\cite{ELW.2010}. To obtain the Schur expansion of a symmetric function $f = \sum_{\alpha \models n} c_\alpha F_\alpha$:
\begin{enumerate}
\item replace $F_\alpha = 0$ for $\alpha$ not a partition;
\item for each partition $\mu$ replace $F_\mu$ by $\sum_c (-1)^{\ell(c)-1} s_{\mathsf{wt}(c)}$ summed over all chains $c$ as in
Definition~\ref{definition.chains} such that the initial shape is $\mu$.
\end{enumerate}

\begin{example}
\def\cc{.7}
Suppose
\[
	f= F_{(1, 2, 2)} + F_{(1, 3, 1)} + F_{(1, 4)} + F_{(2, 2, 1)} + 2 F_{(2, 3)} + 2 F_{(3, 2)} + F_{(4, 1)}.
\]
We can check that $f$ is a symmetric function
by computing the monomial expansion and checking that the coefficient of
$M_\alpha$ is equal to the coefficient of $M_\beta$ when $\beta$
is a permutation of the parts of $\alpha$.
To obtain the Schur expansion, the first step
is to replace $F_\alpha=0$ for all $\alpha$ which are not partitions to obtain:
\[
	F_{(2, 2, 1)} + 2 F_{(3, 2)} + F_{(4, 1)}.
\]
\begin{itemize}
\item
There is only one chain starting with shape $(2,2,1)$, namely
$\left(~\begin{ytableau}
{3}\\
{2}&{2}\\
{1}&{1}\\
\end{ytableau}~\right)$.

\item
There are two chains starting with shape $(3,2)$:
\begin{equation*}
    \left(~
        \begin{ytableau}
            {2}&{2}\\
            {1}&{1}&{1}\\
        \end{ytableau}
    ~\right)
    \qquad \text{and} \qquad
    \left(~
        \begin{ytableau}
            {2}&{3}\\
            {1}&{1}&{2}\\
        \end{ytableau}
        ~, \quad
        \begin{ytableau}
            {3}\\
            {2}&{2}\\
            {1}&{1}\\
        \end{ytableau}
    ~\right).
\end{equation*}

\item
Three chains starting with shape $(4,1)$:
\begin{equation*}
    \left(~
        \begin{ytableau}
            {2}\\
            {1}&{1}&{1}&{2}\\
        \end{ytableau}
        ~, \quad
        \begin{ytableau}
            {2}&{3}\\
            {1}&{1}&{2}\\
        \end{ytableau}
        ~, \quad
        \begin{ytableau}
            {3}\\
            {2}&{2}\\
            {1}&{1}\\
        \end{ytableau}
    ~\right)
    , \quad
    \left(~
        \begin{ytableau}
            {2}\\
            {1}&{1}&{1}&{2}\\
        \end{ytableau}
        ~, \quad
        \begin{ytableau}
            {2}&{2}\\
            {1}&{1}&{1}\\
        \end{ytableau}
    ~\right)
    \quad \text{and} \quad
    \left(~
        \begin{ytableau}
            {2}\\
            {1}&{1}&{1}&{1}\\
        \end{ytableau}
    ~\right).
\end{equation*}
\end{itemize}
Hence the Schur expansion of $f$ is
\[
	f = s_{(2,2,1)} + 2(-s_{(2,2,1)}+s_{(3,2)}) + (s_{(2,2,1)} - s_{(3,2)} + s_{(4,1)}) = s_{(3,2)} + s_{(4,1)}.
\]
\end{example}

\begin{remark}
\label{remark.permutation}
Instead of picking the square matrix $\mathsf{Q}= (\mathsf{QK}_{\lambda,\mu})_{\lambda,\mu}$ for partitions $\lambda,\mu \vdash n$, one can pick a different
set $\mathcal{S} = \{ \sigma_\mu(\mu) \mid \mu \vdash n\}$, where $\sigma_\mu \in S_{\ell(\mu)}$ is a permutation. The square matrix
$(\mathsf{QK}_{\lambda,\alpha})_{\lambda,\alpha}$ with $\lambda \vdash n$ and $\alpha \in \mathcal{S}$ is still invertible and often sparser than $\mathsf{Q}$.
The invertibility follows from the fact that $\mathsf{QK}_{\lambda,\sigma(\lambda)}=1$, which follows from $K_{\lambda, \sigma(\lambda)}=1$ and observing
that this unique tableau is quasi-Yamanouchi (since the entries in the first column contain $1,2,\ldots,\ell(\lambda)$).
All arguments in this section still go through in this setting.
\end{remark}

Following the notation from Corollary~\ref{cor.b=c} and
Theorem~\ref{theorem.main}, we note that
Egge, Loehr and Warrington~\cite{ELW.2010} give the entries
of a matrix $K^\ast_{\alpha, \lambda}$ such that
\begin{equation}
\label{equation.b in terms of c.ELW}
	b_\lambda = \sum_{\alpha \models n} K^\ast_{\alpha, \lambda} c_\alpha~.
\end{equation}
It is worth comparing this result with Equation~\eqref{equation.b in terms of c}.
Their matrix $K^\ast = (K^\ast_{\alpha,\lambda})_{\alpha \models n, \lambda \vdash n}$ is rectangular compared to
our matrix $\mathsf{Q}^{-1} = (\mathsf{Q}^{-1}_{\mu,\lambda})_{\mu,\lambda \vdash n}$, which is square.
Equation~\eqref{equation.b in terms of c} can be interpreted as
a different right inverse of this matrix with the same dimensions as $K^\ast$
where the rows that are indexed by partitions $\mu \vdash n$
are $\mathsf{Q}^{-1}_{\mu,\lambda}$ and the rows indexed by
compositions that are not partitions contain only zero entries.

The matrix $K^\ast$ is sparse since for a fixed
$\alpha$ there is at most one partition $\lambda$ such that $K^\ast_{\alpha,\lambda} \neq 0$
and all the non-zero entries are $\pm 1$.
By comparison, we have that the coefficients $\mathsf{Q}^{-1}_{\mu,\lambda}$
are often in $\{0,1,-1\}$ but may lie out of this set.  We note, however,
that it is easy to deduce from Equation \eqref{equation.Qinverse} that
$\mathsf{Q}^{-1}_{\mu,(n-k,1^k)} = 0$ unless $\mu = (n-k,1^k)$ is also a hook.

\section{Applications}
\label{section.applications}

In this section, we consider several applications of the methods developed in this paper.

\subsection{Leading terms in plethysm of Schur functions and Newton polytopes}
\label{section.leading term}

Finding a combinatorial interpretation for the plethysm coefficients $a_{\lambda,\mu}^\nu$ for the Schur functions
\begin{equation}
\label{equation.plethysm}
	s_\lambda[s_\mu] = \sum_\nu a_{\lambda,\mu}^\nu s_\nu
\end{equation}
is in general an open problem (see for example~\cite{COSSZ.2022}).

Recently, Panova and Zhao~\cite{PZ.2023} asked whether for general partitions $\lambda$ and $\mu$, there exists a partition $\nu$ such that
$a_{\lambda,\mu}^\nu=1$ and $\nu$ is largest in lexicographic order among all partitions $\kappa$ with $a_{\lambda,\mu}^\kappa>0$.
They conjectured that
\begin{equation}
\label{equation.nu}
	\nu = \lambda_1 \mu^{(1)}+\lambda_2 \mu^{(2)} + \cdots + \lambda_\ell \mu^{(\ell)},
\end{equation}
where $\ell$ is the length of $\lambda$ and $\mu^{(1)},\ldots,\mu^{(\ell)}$ are the first $\ell$ terms in reverse lexicographic order in the Newton
polytope of $s_\nu$. More precisely, $\mu^{(1)}=\mu$ and $\mu^{(i)} = (\mu_1,\mu_2,\ldots,\mu_{k-1},\mu_k-1,0^{i-2},1)$ for $i>1$ with $k$ the length
of $\mu$.
Panova and Zhao~\cite{PZ.2023} point out that if $\nu$ is also largest in
dominance order this would imply that the plethysm of two Schur functions
$s_\lambda[s_\mu]$ has a saturated Newton polytope (see~\cite{MPY.2022}).
Paget and Wildon \cite{PW.2019} characterize that the Schur expansion
of $s_\lambda[s_\mu]$ has a unique maximal term in dominance order if and only
if $\ell(\lambda)=1$ or $\mu=(1)$ or $\ell(\lambda)=2$ and $\mu$ is a rectangle.
When there is no unique term in dominance order, it is still open whether 
$s_\lambda[s_\mu]$ has a saturated Newton polytope.

We use the methods in this paper and the
fundamental expansion of $s_\lambda[s_\mu]$ due to
Loehr and Warrington~\cite{LW.2012} to prove this result.
Loehr and Warrington's combinatorial formula for this expression states that
\begin{equation}
\label{equation.plethysm F expansion}
	s_\lambda[s_\mu] = \sum_{T \in \mathsf{SToT}(\lambda,\mu)} F_{\mathsf{iDes}(T)},
\end{equation}
where the sum runs over all standard tableaux of tableaux of outer shape $\lambda$ and inner shape $\mu$. In a tableau of tableaux the entries in
the boxes of $\lambda$ are themselves tableaux of shape $\mu$ ordered in lexicographic order by row reading word.
For each $T$ which is a standard tableau of tableaux, $\mathsf{iDes}(T)$ is
the inverse descent composition of the \defn{dynamic reading word} of $T$.

Let us explain how to obtain the dynamic reading word from $T \in \mathsf{SToT}(\lambda,\mu)$. First we construct a $|\lambda|\times |\mu|$-dimensional
matrix $A$ by associating a row in $A$ to each box in $\lambda$, reading top to bottom, left to right. The row in $A$ associated to a cell $C$ in $\lambda$
is obtained by the row reading word from the tableau of shape $\mu$ inside cell $C$ (again top to bottom, left to right).

The dynamic reading word is a permutation $w$ and $\mathsf{iDes}(T)$ is the composition $\alpha$ associated with the inverse descent set
$\{ i \mid w^{-1}(i) > w^{-1}(i+1)\}$.

\begin{example}
\label{example.tot}
Consider the tableau of tableaux
\[
T = \scalebox{.8}{\tableau[l]{\tableau[sY]{a&b\\1&4} \\ \tableau[sY]{5&7\\ 2&3} & \tableau[sY]{8&c\\ 6&9}}\;}
 \in \mathsf{SToT}((2,1),(2,2)),
\]
where we use the alphabet $\{1<2<\cdots<9<a<b<c\}$. The matrix $A$ is given by
\[
	A = \begin{bmatrix}
	a&b&1&4\\
	5&7&2&3\\
	8&c&6&9
	\end{bmatrix}.
\]
\end{example}

Next scan the columns of $A$ from left to right. For each column index $k < |\mu|$, use the symbols in column $k+1$ of $A$ to determine a total ordering
of the rows in which the row with the smallest symbol in column $k+1$ comes first, and the row with the largest symbol in column $k+1$ comes last.
Write the symbols in column $k$ using the row ordering just determined. For the rightmost column, write the symbols in order from top to bottom.
The resulting word is $\mathsf{dynamic}(T)$ (see~\cite{LW.2012}).

\begin{example}
For $T$ in Example~\ref{example.tot}, we obtain $\mathsf{dynamic}(T) = 5a8b7c216439$ and $\mathsf{iDes}(T)=(1,2,1,2,1,2,3)$.
\end{example}

By Corollary~\ref{cor.b=c}, the coefficient of the lexicographic largest term $F_\nu$ in~\eqref{equation.plethysm F expansion} is the same as
the coefficient of $s_\nu$ in the Schur expansion of $s_\lambda[s_\mu]$. Let $n=|\lambda|$. We obtain $\nu$ by filling the last $\mu_1$ columns
of $A$ from top to bottom, left to right with the numbers $1,2,\ldots, n\mu_1$, the preceding $n\mu_2$ columns with the numbers
$n\mu_1+1, n\mu_1+2,\ldots,n(\mu_1+\mu_2)$ and so on, until we get to the last part $\mu_k$ of $\mu$.
If $\mu_k > 1$, fill columns 2 to $\mu_k$ with
$n(\mu_1+\cdots+\mu_{k-1}+1)+1,n(\mu_1+\cdots+\mu_{k-1}+1)+2,\ldots,n(\mu_1+\cdots+\mu_k)$.
Note that the inverse descent composition associated with this (partially-filled) matrix is
$(n\mu_1,\ldots,n\mu_{k-1},n\mu_k-n)$, which is the lexicographically largest that can be achieved
and can only be achieved in this way.

The first column is filled with the remaining $n$ numbers $n(\mu_1+\cdots+\mu_{k-1})+1,n(\mu_1+\cdots+\mu_{k-1})+2,\ldots,n(\mu_1+\cdots+\mu_{k-1})+n$ as follows. Note that they cannot be
filled in order from top to bottom since then the corresponding tableau of tableaux would not be column-strict
(unless $\lambda$ is a single row). To obtain the largest increasing string of numbers with the previously $n\mu_k-n$ placed numbers
$n(\mu_1+\cdots+\mu_{k-1}+1)+1,n(\mu_1+\cdots+\mu_{k-1}+1)+2,\ldots,n(\mu_1+\cdots+\mu_k)$, we can adjoin at most $\lambda_1$
numbers due to the shape of $\lambda$. This is achieved by the reading word of the standardization of the
unique tableau of shape $\lambda$ and weight $(\lambda_2,\ldots, \lambda_\ell,\lambda_1)$ with all letters shifted by $n(\mu_1+\cdots+\mu_{k-1})$.
Note that this tableau is unique since $K_{\lambda,\sigma{\lambda}}=1$ for any permutation $\sigma \in S_{\ell(\lambda)}$.

In this construction, all columns of $A$ except for the first column are increasing from top to bottom. Hence the dynamic reading word is just the column
reading word of $A$. The inverse descent composition
is $\tilde{\nu}=(n\mu_1,\ldots,n\mu_{k-1},\lambda_2,\ldots,\lambda_\ell,n\mu_k-n+\lambda_1)$ and $A$
with this inverse descent composition is unique. Hence $F_{\tilde{\nu}}$ appears in
$s_\lambda[s_\mu]$ with coefficient 1.
Note that $\tilde{\nu}$ is generally a composition and not a partition. However, from Remark~\ref{remark.permutation}, we also have that
$F_\nu$ with
\begin{equation}
\label{equation.nu new}
	\nu = (n\mu_1,\ldots,n\mu_{k-1},n\mu_k-n+\lambda_1,\lambda_2,\ldots,\lambda_\ell)
\end{equation}
appears in $s_\lambda[s_\mu]$ with coefficient 1.
This gives an alternative proof of~\cite[Corollary 9.1]{PW.2019}
and~\cite[Theorem 4.2]{I.2011}, which we state in the following proposition.

\begin{proposition} \label{th:leadingcoefficient}
For partitions $\lambda \vdash n$ and $\mu$,
the partition $\nu$ that is largest in lexicographic order such that $a^\nu_{\lambda,\mu}\neq 0$ in~\eqref{equation.plethysm} is as in~\eqref{equation.nu new}.
Moreover, $a^\nu_{\lambda,\mu}=1$.
\end{proposition}

The partition $\nu$ in~\eqref{equation.nu new} agrees with~\eqref{equation.nu} after some rewriting.

\begin{example}
For $\lambda=(2,1)$ and $\mu=(2,2)$, this construction gives the tableau of tableaux
\[
T = \scalebox{.8}{\tableau[l]{\tableau[sY]{8&a\\1&4} \\ \tableau[sY]{7&b\\ 2&5} & \tableau[sY]{9&c\\ 3&6}}\;}
\qquad \text{with matrix} \qquad
	A = \begin{bmatrix}
	8&a&1&4\\
	7&b&2&5\\
	9&c&3&6
	\end{bmatrix}
\]
and $\mathsf{dynamic}(T) = 879abc123456$ and $\mathsf{iDes}(T)=(6,1,5)$. Indeed, we have
$s_{(2,1)}[s_{(2,2)}] = s_{(6,5,1)}+$ terms strictly smaller in lexicographic order.
\end{example}

A similar approach can be used to determine the second largest
partition $\kappa$ in lexicographic order such that $a_{\lambda, \mu}^\kappa >0$.
We use the fact observed in this paper that only the coefficients of
$F_\alpha$ with $\alpha$ a partition are needed. For example,
for $\mu_k>1$ with $k=\ell(\mu)$ and $\ell>2$,
we deduce that this second largest partition is
\[
	\kappa = (n\mu_1,\ldots,n\mu_{k-1},n\mu_k-n+\lambda_1-1,\lambda_2+2,\lambda_3,\ldots,\lambda_{\ell-1},\lambda_\ell-1)
\]
and that $a_{\lambda,\mu}^\kappa = 1$. Since $\kappa$ is incomparable to $\nu$ in~\eqref{equation.nu new} in dominance order, there is no
maximal term in dominance order in $s_\lambda[s_\mu]$ in this case. This agrees with the classification in~\cite{PW.2019}.

\subsection{Plethysm of Schur functions in two variables}
\label{section.two variables}

By~\eqref{equation.finite variables}, we can restrict to $m$ variables. Doing so for the plethysm $s_w[s_h](x,y)$
reduces~\eqref{equation.plethysm F expansion} to
\begin{equation}
\label{equation.wh}
	s_w[s_h](x,y) = \sum_{T \in \mathsf{SToT}((w),(h))} F_{\mathsf{iDes}(T)}(x,y).
\end{equation}
To convert this into a Schur expansion, by Theorem~\ref{theorem.main} it suffices to restrict to summing over $T$ such that $\mathsf{iDes}(T)$ is a
partition with at most two parts
\[
	S_{w,h} = \sum_{\substack{T \in \mathsf{SToT}((w),(h))\\ \mathsf{iDes}(T)=(\lambda_1,\lambda_2) \text{ partition of $wh$}}}
	F_{(\lambda_1,\lambda_2)}(x,y).
\]
By the results of Loehr and Warrington~\cite{LW.2012} described in Section~\ref{section.leading term}, $\mathsf{iDes}(T)$ has at most two parts, if $T$
is the standardization of a tableau of tableaux in the alphabet $\{1,2\}$.

\begin{example}
For $h=4$ and $w=5$, the matrix $M$ below comes from a tableau of tableaux in the alphabet $\{1,2\}$.
\[
	M = \begin{bmatrix}
    1&1&1&1\\\cline{4-4}
    1&1&1&\multicolumn{1}{|c}{2}\\
    1&1&1&\multicolumn{1}{|c}{2}\\\cline{2-3}
    1&\multicolumn{1}{|c}{2}&2&2\\\cline{1-1}
	2&2&2&2
	\end{bmatrix}
	\qquad \text{and standardizes to} \qquad
	A=\begin{bmatrix}
	1&5&8&11\\
	2&6&9&17\\
	3&7&10&18\\
	4&13&15&19\\
	12&14&16&20
	\end{bmatrix}.
\]
The inverse descent composition of the tableau of tableaux $T$ corresponding to $A$ is $\mathsf{iDes}(T) = (\lambda_1,\lambda_2)=(11,9)$.
Note that the border between the 1's and 2's in $M$ defines a partition $\mu$ of length at most $w$  and $\mu_1\leqslant h$, where $\mu_i$
is the number of 2's in the $i$-th column, from the right. In this example $\mu=(4,2,2,1)$.
\end{example}

Let $L(w, h)$ be the set of partitions $\mu$ in a $w \times h$ box, that is, $\mu_1 \leqslant w$ and $\ell(\mu) \leqslant h$.
Note that the dynamic reading word of the matrix $A$ corresponding to the partition $\mu$ with $\mu=(w^k,a)$ is $123\ldots (wh)$ for all
$0\leqslant k\leqslant h$ and $0\leqslant a <w$.
For example, for $\mu = (4,4,2)$,
\[
	\begin{bmatrix}
    \cline{3-4}
    1&1&\multicolumn{1}{|c}{2}&2\\
    1&1&\multicolumn{1}{|c}{2}&2\\\cline{2-2}
    1&\multicolumn{1}{|c}{2}&2&2\\
    1&\multicolumn{1}{|c}{2}&2&2\\
    \cline{1-1}
	\end{bmatrix}
	\quad \text{standardizes to} \quad
	\begin{bmatrix}
    1 & 5 & 9 & 13 \\
    2 & 6 & 10 & 14 \\
    3 & 7 & 11 & 15 \\
    4 & 8 & 12 & 16
	\end{bmatrix}.
\]
The dynamic reading words for all other partitions
are distinct and their inverse descent composition is $(\lambda_1,\lambda_2)$ with $\lambda_1$ the number of 1's in $M$ and $\lambda_2$
the number of 2's in $M$. Hence
\begin{equation}
\label{equation.Swh}
	S_{w,h} = F_{(wh)} + \sum_{\substack{\lambda=(\lambda_1,\lambda_2) \vdash wh\\ \lambda_2>0}} c_\lambda F_\lambda,
\end{equation}
where
\begin{equation}\label{equation.clambda}
    c_\lambda = \# \{ \mu \vdash \lambda_2 \mid \mu \in L(w, h) \text{~and~} \mu \text{ not of the form $(w^k,a)$}\}.
\end{equation}

Since we are working with two variables,
by Corollary~\ref{corollary-submatrix} it suffices to consider the submatrix
$\mathsf{Q}_{\leqslant 2}$ of the quasi-Kostka matrix indexed by partitions of
$wh$ with at most two parts. The submatrix indexed by the partitions with
exactly two parts is the unit upper-triangular matrix with all nonzero entries
equal to one. Its inverse has $1$ on the diagonal, $-1$ in row $i$ and column
$i+1$, and 0 elsewhere.
\begin{example}
For $wh=8$ and ordering the partitions of $8$ with at most two parts in reverse lexicographic order $\{(8),(7,1),(6,2),(5,3),(4,4)\}$, we have
\[
	\mathsf{Q}_{\leqslant 2} = \begin{pmatrix}
	1&0&0&0&0\\
	0&1&1&1&1\\
	0&0&1&1&1\\
	0&0&0&1&1\\
	0&0&0&0&1
	\end{pmatrix}
	\quad \text{and} \quad
	\mathsf{Q}_{\leqslant 2}^{-1} = \begin{pmatrix}
	1&0&0&0&0\\
	0&1&-1&0&0\\
	0&0&1&-1&0\\
	0&0&0&1&-1\\
	0&0&0&0&1
	\end{pmatrix}.
\]
\end{example}

Hence by Theorem~\ref{theorem.main}, we obtain the Schur expansion
\begin{equation}
\label{equation.b lambda}
	s_w[s_h](x,y) = \sum_{\lambda=(\lambda_1,\lambda_2) \vdash wh} a^\lambda_{(w),(h)} s_\lambda(x,y),
\end{equation}
where $a^{(wh)}_{(w),(h)}=1$ and for $\lambda = (\lambda_1, \lambda_2)$ with $\lambda_2>0$,
\begin{equation}
\label{equation.b 2 variable}
	a^\lambda_{(w),(h)} = c_{(\lambda_1,\lambda_2)} - c_{(\lambda_1+1,\lambda_2-1)}.
\end{equation}

\subsubsection{The case $w=2$}
For $w=2$, Equation~\eqref{equation.Swh} simplifies to
\[
	S_{2,h} = F_{(2h)} + \sum_{m\geqslant 2} \left\lfloor \frac{m}{2} \right\rfloor F_{(2h-m,m)}.
\]
Using~\eqref{equation.b 2 variable}, we rederive the well-known formula
\begin{equation}
	s_2[s_h] = \sum_{\substack{0\leqslant c \leqslant h\\ \text{$c$ even}}} s_{(2h-c,c)}.
\end{equation}
This analysis is much easier than the one presented in~\cite{LW.2012}.

\subsubsection{Symmetric chain decompositions}
For general $w$, a combinatorial description for $a^\lambda_{(w),(h)}$ can be found by determining a \defn{symmetric chain decomposition} (SCD)
in the subposet $L(w,h)$ of Young's lattice restricted to partitions in a box of height $h$ and width $w$. Concretely, for
$\lambda,\mu \in L(w,h)$ we have $\lambda \leqslant \mu$ if and only if $\lambda_i \leqslant \mu_i$ for all $1\leqslant i \leqslant h$. In fact, $L(w, h)$ is a
graded lattice, with \defn{rank function}
\begin{equation*}
    \rank(\lambda) = \lambda_1 + \cdots + \lambda_h.
\end{equation*}
A symmetric chain decomposition is a partitioning of $L(w,h)$ into saturated, rank-symmetric chains.
In this setting, the formula for $a^\lambda_{(w),(h)}$ in~\eqref{equation.b 2 variable} can be restated as
\begin{equation}
\label{equation.b L}
	a^\lambda_{(w),(h)} = \#\{ \mu \in L(w,h) \mid \rank(\mu)=\lambda_2\} - \# \{ \mu \in L(w,h) \mid \rank(\mu) = \lambda_2-1\},
\end{equation}
since for any $\lambda_2$ there is precisely one partition of $\lambda_2$ of the form $(w^k,a)$
and so by \eqref{equation.clambda} we have $\#\{ \mu \in L(w,h) \mid \rank(\mu)=\lambda_2\} = c_\lambda + 1$.

In the language of symmetric chain decompositions, the saturated chains provide an injection from the elements of rank $\lambda_2-1$ to those of
rank $\lambda_2$. Therefore $a^\lambda_{(w),(h)}$ is the number of symmetric chains that begin at rank $\lambda_2$.

Stanley~\cite{Stanley1980} proved that $L(w,h)$ is rank-symmetric, rank-unimodal, and exhibits the Sperner property.
Furthermore, it was conjectured in~\cite{Stanley1980} that $L(w, h)$ admits a symmetric chain decomposition.
O'Hara~\cite{OHara1990} proved that a related poset structure
$\tilde{L}(w, h)$ admits a symmetric chain decomposition, where $\tilde{L}(w,h)$
is the partial order on the partitions in a box of height $h$ and width $w$ defined by
$\lambda \leqslant_H \mu$ if and only if $\rank(\lambda) \leqslant \rank(\mu)$.
Note that $\tilde{L}(w,h)$ has many more covering relations than $L(w,h)$. O'Hara's method is recursive and does not
yield explicit descriptions for the minimal elements in each chain.
Lindstrom~\cite{Lindstrom1980} gave a symmetric chain decomposition for $L(3, h)$.
West~\cite{West1980} gave a symmetric chain decomposition for $L(4, h)$.
Ries~\cite{Riess1978} and Wen~\cite{Wen2004} gave symmetric chain decompositions for $L(3, h)$ and $L(4, h)$.
Recently, Wen~\cite{Wen.2024} gave a computer assisted proof of a symmetric chain decomposition of $L(5,h)$.

Here we give new symmetric chain decompositions for $w=3,4$.
One advantage of these decompositions is that we can explicitly describe the minimal elements in each chain, which gives a combinatorial
expression for $a^\lambda_{(w),(h)}$.
In addition, these symmetric chain decompositions satisfy the following three conditions.

\smallskip

\noindent \textbf{The restriction condition.}
We are interested in symmetric chain decompositions of $L(w, h)$ that
``restrict'' to symmetric chain decompositions of $L(w, h-1)$ in the following
sense. Denote a saturated chain in $L(w, h)$ by
\begin{equation*}
    \lambda^{(0)} \lessdot \lambda^{(1)} \lessdot \cdots
    \lessdot \lambda^{(c-1)}
    \lessdot \lambda^{(c)}.
\end{equation*}
The \defn{restriction} of this chain to $L(w, h-1)$ is the chain
obtained by removing the elements that do not belong to $L(w, h-1)$.
Thus, there exists a $d$ between $0$ and $c$ such that
\begin{equation*}
    \underbrace{
        \lambda^{(0)} \lessdot \lambda^{(1)} \lessdot \cdots
        \lessdot \lambda^{(d)}
    }_{\in L(w, h-1)}
    ~\lessdot
    \underbrace{ \lambda^{(d+1)} }_{\notin L(w, h-1)}
    \lessdot
    \cdots
    \lessdot
    \underbrace{ \lambda^{(c-1)} }_{\notin L(w, h-1)}
    \lessdot
    \underbrace{ \lambda^{(c)} }_{\notin L(w, h-1)}
\end{equation*}
so that the restriction of the chain is
$\lambda^{(0)} \lessdot \lambda^{(1)} \lessdot \cdots \lessdot \lambda^{(d)}$.

We say that a SCD $\mathscr C$ of $L(w, h)$ \defn{restricts to a SCD of $L(w,
h-1)$} if the set of chains obtained by restricting all the chains in $\mathscr
C$ to $L(w, h-1)$ forms a SCD of $L(w, h-1)$.

\smallskip

\noindent \textbf{The extension condition.}
We say that a SCD $\mathscr C$ of $L(w, h)$ satisfies the \defn{extension
condition} if, for every chain of $C \in \mathscr C$,
the maximal element of $C$ and the maximal element of the restriction of $C$ to $L(w, h-1)$
have the same complement. The \defn{complement} of $\lambda$ in $L(w, h)$ is
\begin{equation*}
     (w - \lambda_{h}, w - \lambda_{h-1}, \ldots, w - \lambda_{1}).
\end{equation*}
For example, the complement of $\lambda = (2, 1)$ in $L(4,3)$ is $(4,3,2)$.
Explicitly, if the restriction of
$\lambda^{(0)} \lessdot \cdots \lessdot \lambda^{(c)}$
to $L(w, h-1)$ is
$\lambda^{(0)} \lessdot \cdots \lessdot \lambda^{(d)}$,
then $(w^h) - \lambda^{(c)} = (w^{h-1}) - \lambda^{(d)}$ viewed as vectors.

\smallskip

\noindent \textbf{The pattern condition.}
Consider a cover relation $\lambda^{(i-1)} \lessdot \lambda^{(i)}$ in $L(w,h)$.
Viewing the elements of $L(w,h)$ as vectors, note that the difference
$(\lambda^{(i)})^t - (\lambda^{(i-1)})^t$ is equal to $\vec e_j
= (0,\ldots,0,1,0,\ldots,0)$ for some $1\leqslant j \leqslant w$, where the unique $1$ appears in
position $j$. The \defn{edge labelling} of
\begin{equation*}
    \lambda^{(0)} \lessdot \lambda^{(1)} \lessdot \cdots
    \lessdot \lambda^{(\ell-1)}
    \lessdot \lambda^{(\ell)}
\end{equation*}
is the sequence $(c_1, \ldots, c_\ell)$ defined by the condition $(\lambda^{(i)})^t- (\lambda^{(i-1)})^t = \vec e_{c_i}$.
We say a chain is \defn{described by a pattern} if its edge labelling
is of the form
\begin{equation*}
    \Big(
        \underbrace{c_1, c_2, \ldots, c_j}_{\text{no $1$s}},
        \overbracket{
            \underbrace{c_{j+1}, \ldots, c_{j+w}}_{\text{permutation}},
            \ldots,
            \underbrace{c_{j+1}, \ldots, c_{j+w}}_{\text{permutation}}
        }^{\text{repeating part}}
    \Big).
\end{equation*}
The condition $c_i \neq 1$ for $1 \leqslant i \leqslant j$ says that adding boxes in
columns $c_1, \ldots, c_j$ of $\lambda^{(0)}$ does not increase the height of
the first column.

Finally, we say that a SCD satisfies the \defn{pattern condition} if each of
its chains is described by a pattern. This implies the extension condition
since the complementary shape of $(\lambda^{(0)})^t + \sum_{1 \leqslant i \leqslant j} \vec e_{c_i}$
is equal to the complementary shape of
$\big((\lambda^{(0)})^t + \sum_{1 \leqslant i \leqslant j} \vec e_{c_i}\big) + m \sum_{1 \leqslant i \leqslant w} \vec e_{c_{j+i}}$
for all $m$.

In Sections~\ref{section.w=3} and~\ref{section.w=4}, we construct symmetric chain decompositions satisfying the restriction, extension and
pattern condition for $w=3,4$. For $w=2$, such a symmetric chain decomposition also exists: in this case,
the lowest weight elements are of the form $1^{m_1}$ with $m_1$ even and boxes
are added in alternating fashion in column 1 and column 2, and there is
a unique pattern.

\begin{conjecture}
Symmetric chain decompositions satisfying the restriction, extension and pattern conditions exist for all $L(w, h)$.
\end{conjecture}

\subsubsection{The case $w=3$}
\label{section.w=3}

We give a symmetric chain decomposition of $L(3,h)$ satisfying the restriction, extension, and pattern conditions.
This is different than symmetric chain decompositions for $L(3, h)$ that have appeared in the literature
\cite{Lindstrom1980, Riess1978, Wen2004}.
One advantage of our symmetric chain decomposition is that we can explicitly describe the minimal and maximal elements in each chain.
As a consequence we get an explicit combinatorial description for $a^\lambda_{(w),(h)}$ in~\eqref{equation.b lambda} when $w=3$
as given in Corollary~\ref{cor.b w=3}.

\begin{theorem}
\label{theorem.w=3}
    Let $h \in \NN$.
    \begin{enumerate}
        \item
            There exists a symmetric chain decomposition $\mathscr{C}$ of
            $L(3, h)$ satisfying the restriction and pattern condition
            (and hence the extension condition).

        \item
            $\lambda$ is a minimal element of a chain in $\mathscr{C}$ iff
            $\lambda \in L(3, h)$ and $\lambda= 3^{m_3} 1^{m_1}$ with
            \begin{equation*}
                {m_1} \geqslant 3{m_3}
                \qquad\text{and}\qquad
                {m_1} \neq 3{m_3} + 1;
            \end{equation*}
            or equivalently, $\lambda = 3^{m_3} 1^{m'_1 + 3 m_3}$ with
            ${m'_1} \geqslant 0$ and ${m'_1} \neq 1$.

        \item
            If the minimal element in a chain $C \in \mathscr{C}$ is of the form
            $\lambda = 3^{m_3}1^{m_1 + 3m_3}$ with $m_1\neq1$,
            then the maximal element in $C$ is
            \begin{enumerate}
            \item $3^{h-m_1-3m_3} 2^{m_1}1^{3m_3}$ if $m_1$ is even;
            \item $3^{h-m_1-3m_3+1} 2^{m_1-2} 1^{3m_3+1}$ if $m_1$ is odd.
            \end{enumerate}

        \item
            $\lambda$ is a maximal element of a chain in $\mathscr{C}$ iff
            its complementary shape is of the form $2^{m_2} 1^{m_1}$ with
            \begin{itemize}[noitemsep]
                \item ${m_1} \leqslant -\frac{4}{3} {m_2} + h$;
                \item $m_1 \equiv 0 \pmod 2$ iff $m_2 \equiv 0 \pmod 3$;
                \item $m_1 \equiv 1 \pmod 2$ iff $m_2 \equiv 1 \pmod 3$.
            \end{itemize}
    \end{enumerate}
\end{theorem}

The proof of Theorem~\ref{theorem.w=3} is given after the next corollary by decomposing $L(3,h)$ into smaller components
and then giving a symmetric chain decomposition of the smaller component in Proposition~\ref{proposition.L'} below.

\begin{corollary}
\label{cor.b w=3}
We have $s_3[s_h](x,y) = \sum_{\lambda=(\lambda_1,\lambda_2) \vdash 3h} a^\lambda_{(3),(h)} s_\lambda(x,y)$ with
\[
	a^\lambda_{(3),(h)} = \#\{ \mu \vdash \lambda_2 \mid \mu=3^{m_3} 1^{m_1} \text{ such that } m_1+m_3\leqslant h,
	m_1 \geqslant 3m_3 \text{ and } m_1 \neq 3m_3+1\}.
\]
\end{corollary}

\begin{proof}
By~\eqref{equation.b L}, we have $a^\lambda_{(3),(h)} = \#\{\mu \in L(3,h) \mid \rank(\mu)=\lambda_2\} -
\#\{\mu \in L(3,h) \mid \rank(\mu) = \lambda_2-1\}$, which is the number of minimal elements at rank $\lambda_2$ in a symmetric chain
decomposition of $L(3,h)$. The result follows from Theorem~\ref{theorem.w=3} (2).
\end{proof}

\begin{definition}
We say that a partition $\lambda = \ell^{m_\ell} \cdots 2^{m_2} 1^{m_1}$ \defn{contains $\mu$ as a
subpartition} if $\mu = \ell^{a_\ell} \cdots 2^{a_2} 1^{a_1}$
and $a_i \leqslant m_i$ for all $1 \leqslant i \leqslant \ell$. In this case, we write
$\lambda = \mu \oplus \nu$, where $\nu = \ell^{m_\ell-a_\ell} \cdots 2^{m_2-a_2} 1^{m_1-a_1}$.
\end{definition}

Define $L'(3,h)$ to be the partitions $\lambda=3^{m_3} 2^{m_2} 1^{m_1}$ in $L(3,h)$ with either $m_3=0$ or $m_1\leqslant 2$.
That is $L'(3,h)$ is the set of partitions that do not contain $3111$ as a subpartition.
If $L'(3,h)$ has a symmetric chain decomposition,
then $L(3,h)$ also has a symmetric chain decomposition since
\begin{equation}
\label{equation.L decomp}
	L(3,h) = L'(3,h) \uplus \{ \lambda\oplus(3,1,1,1) \mid  \lambda \in L(3,h-4)\}.
\end{equation}
The decomposition in~\eqref{equation.L decomp} is related to~\cite[Theorem 2.1]{Tetreault2020} and~\cite[Lemma 5.5]{COSSZ.2022}.

Note that the set of partitions $\{ \lambda\oplus(3,1,1,1)\mid \lambda \in L(3,h-4)\}$
is isomorphic to $L(3,h-4)$ as a poset. Hence by induction on $h$, we can assume
that $L(3,h-4)$ has a symmetric chain decomposition with the stated properties.  Since the chains of
$L'(3,h)$ and $\{ \lambda\oplus(3,1,1,1) \mid \lambda \in L(3,h-4)\}$ are both
centered at partitions of size $\lceil 3h/2\rceil = \lceil 3(h-4)/2 \rceil + 6$,
the symmetric chain decomposition of $L'(3,h) \uplus \{ \lambda\oplus(3,1,1,1) \mid \lambda \in L(3,h-4)\}$ is
a symmetric chain decomposition of $L(3,h)$.
Hence the proof of Theorem~\ref{theorem.w=3} follows from the following proposition.

\begin{proposition}
\label{proposition.L'}
There is a symmetric chain decomposition $\mathscr C$ of $L'(3,h)$
satisfying:
\begin{enumerate}
\item $\mathscr C$ satisfies the restriction and pattern condition.
\item The minimal element of a chain $C$ in $\mathscr C$ is of the form $1^{m_1}$ with $m_1 \neq 1$.
\item  If the minimal element of a chain $C$ in $\mathscr C$ is $1^{m_1}$, then the corresponding maximal element in $C$ is
\begin{enumerate}
\item $3^{h-m_1}2^{m_1}$ if $m_1$ is even;
\item $3^{h-m_1+1}2^{m_1-2}1$ if $m_1$ is odd.
\end{enumerate}
\end{enumerate}
\end{proposition}

\begin{proof}
A partition $\lambda \in L'(3,h)$ is either a minimal element in a chain (denoted by $\lambda \in \mathsf{HW}$ if $\lambda = 1^{m_1}$
with $m_1\neq 1$), or a maximal element (denoted by $\lambda \in \mathsf{LW}$ if $\lambda = 3^{h-m_2}2^{m_2}$
with $m_2$ even or $\lambda = 3^{h-m_2-1}2^{m_2}1$ if $m_2$ is odd), or it has both a predecessor and a successor in its chain.
To this end, we define two functions
\[
	f \colon L'(3,h) \backslash \mathsf{LW} \rightarrow \{1,2,3\} \quad \text{and} \quad
	e \colon L'(3,h) \backslash \mathsf{HW} \rightarrow \{1,2,3\}
\]
which, if $\lambda^{(c)} \lessdot \lambda^{(c+1)}$ in the symmetric chain decomposition,
then $f(\lambda^{(c)})= e(\lambda^{(c+1)})$ is equal to the column index of the cell $\lambda^{(c+1)}\backslash\lambda^{(c)}$.
We can view $f$ and $e$ as $\mathfrak{sl}_2$ lowering and raising operators. Hence we call the elements in $\mathsf{HW}$ (resp. $\mathsf{LW}$)
\defn{highest weight elements} (resp. \defn{lowest weight elements}).
Note that when $\lambda\in \mathsf{LW}$ and $m_2$ is odd, then $m_2\leqslant h-1$ since $h-m_2-1\geqslant 0$.

Say that a partition $\lambda=3^{m_3} 2^{m_2} 1^{m_1}$ is in \defn{phase~1} if $m_3=0$ and either
$m_1+m_2$ is even and $m_1\geqslant 1$
or $m_1+m_2$ is odd and $m_1>2$; otherwise we say that $\lambda$ is in \defn{phase~2}.

We now define $f$ and $e$. Any partition $\lambda \in L'(3,h)$ is either
\begin{enumerate}
\item 
in phase 1 and of odd length with $m_1>2$.  In this case, $f(\lambda)=2$
and $e(\lambda)=2$ if $m_2>0$ (if $m_2=0$, then it is of highest weight).\\
The highest weight in the chain is $1^{m_1+m_2}$,\\ the lowest weight
in the chain is $3^{h-(m_1+m_2)+1}2^{m_1+m_2-2}1$.
\item  
in phase 1 and of even length.  In this case, $f(\lambda)=2$
and $e(\lambda)=2$ if $m_2>0$ (if $m_2=0$, then it is of highest weight).\\
The highest weight in the chain is $1^{m_1+m_2}$,\\ the lowest weight
in the chain is $3^{h-(m_1+m_2)}2^{m_1+m_2}$.
\item 
in phase 2, $m_1 =0$ and $m_2$ is even. In this case, $f(\lambda)=1$ if $\ell(\lambda)<h$ and $e(\lambda)=3$
(if $m_3=0$, then $e(\lambda)=2$ or $\lambda$ is empty and highest weight).\\
The highest weight in the chain is $1^{m_2}$,\\
the lowest weight in the chain is $3^{h-m_2}2^{m_2}$.
\item 
in phase 2, $m_1=0$ and $m_2$ is odd. In this  case, $f(\lambda)=3$ and $e(\lambda)=2$.\\
The highest weight in the chain is $1^{m_2-1}$,\\
the lowest weight in the chain is $3^{h-m_2+1}2^{m_2-1}$.
\item 
in phase 2, $m_1=1$ and $m_2$ is even. In this case, $f(\lambda)=2$ and $e(\lambda)=1$.\\
The highest weight in the chain is $1^{m_2}$,\\
the lowest weight in the chain is $3^{h-m_2}2^{m_2}$.
\item 
in phase 2, $m_1 =1$ and $m_2$ is odd. In this case, $f(\lambda)=1$ if $\ell(\lambda)<h$ and $e(\lambda)=2$.\\
The highest weight in the chain is $1^{m_2+2}$,\\
the lowest weight in the chain is $3^{h-m_2-1}2^{m_2}1$.
\item 
in phase 2, $m_1 =2$ and $m_2$ is even. In this case, $f(\lambda)=2$ and $e(\lambda)=3$.\\
The highest weight in the chain is $1^{m_2+3}$,\\
the lowest weight in the chain is $3^{h-m_2-2}2^{m_2+1}1$.
\item 
in phase 2, $m_1 =2$ and $m_2$ is odd. In this case, $f(\lambda)=3$ and $e(\lambda)=1$
(unless $m_3=0$ in which case $e(\lambda)=2$).\\
The highest weight in the chain is $1^{m_2+2}$,\\
the lowest weight in the chain is $3^{h-m_2-1}2^{m_2}1$.
\end{enumerate}
One can check explicitly, that $f(e(\lambda))=\lambda$ if $e(\lambda)$ is defined and $e(f(\lambda))=\lambda$ if $f(\lambda)$ is defined. Hence they
are partial inverses of each other and indeed define a symmetric chain decomposition.
If the highest weight element is $1^{m_1}$ with $m_1$ even (resp. odd), then eventually the chain defined by the operator $f$ will oscillate between
cases $(3)\to (5) \to (4) \to (3) \to \cdots$ (resp. $(8) \to (7) \to (6) \to (8) \to \cdots$) with pattern 123 (resp. 321). Hence the chains satisfy the pattern condition.
The restriction condition follows by construction.
\end{proof}

\begin{figure}[htpb]
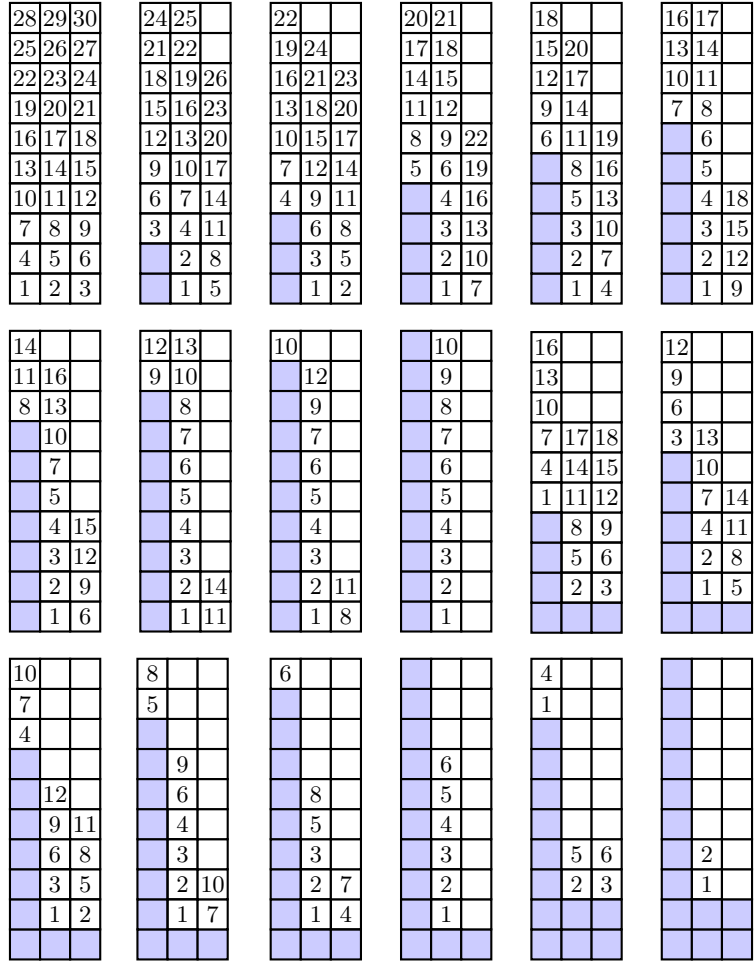

    \centering
    \begin{equation*}

    \end{equation*}
    \caption{The symmetric chain decomposition for $L(3,10)$ satisfying
        the restriction and pattern conditions.  The blue cells are the partitions
        representing the highest weights.  The number entries indicate in which order boxes are added for the given chain.}
    \label{fig:SCD-w3-C0}
\end{figure}

\begin{remark}
Note that by the decomposition~\eqref{equation.L decomp}, the chains are separated into strata. The `outer shell' are all those partitions of the form
$3^{m_3}2^{m_2}1^{m_1}$ with $m_3=0$ or $m_1\leqslant 2$ and the `inner shell' are those with $m_3>0$ and $m_1\geqslant 3$.
In other words we are stratifying the partitions in $L(3,h)$ by the number of copies
of the partition $3111$ they contain.
\end{remark}

\begin{remark}
The maps $f$ and $e$ in the proof of Proposition~\ref{proposition.L'} can be interpreted as $\mathfrak{sl}_2$ crystal operators.
Theorem~\ref{theorem.w=3} (2) and (4) give the explicit highest and lowest elements in this $\mathfrak{sl}_2$ crystal.
\end{remark}

\begin{example}
The symmetric chain decomposition for $L(3,10)$ is given in Figure~\ref{fig:SCD-w3-C0}.
In Figure~\ref{figure.3d} we plot the points $(m_1, m_2, m_3)$ in $3d$ for each partition $3^{m_3} 2^{m_2} 1^{m_1}$ in $L(3,10)$.
We color the edges between the points {\it blue} if the highest weight is $1^{2r+1}$ and we just add cells in column 2,
{\it red} if the highest weight is $1^{2r}$ and we just add cells in column 2,
{\it magenta} if the highest weight is $1^{2r}$ and we fill with pattern $123$ and {\it orange}
if the highest weight is $1^{2r+1}$ and we fill with pattern $321$.
The highest weights are the black dots and the lowest weights are the blue dots.
Notice that a general chain begins with blue edges and finishes with orange edges or it
begins with red edges and finishes with magenta edges.
\begin{figure}[b]
\includegraphics[width=3in]{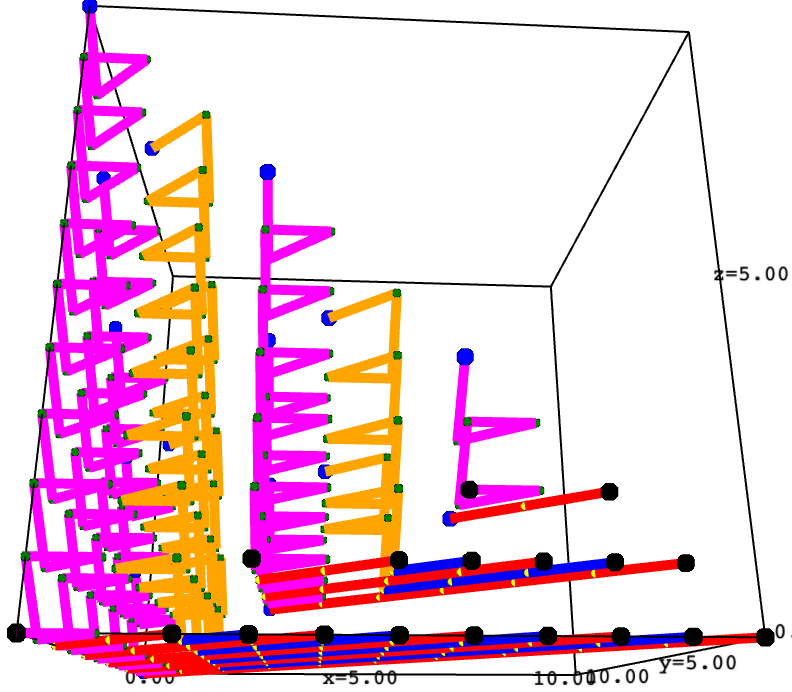}
    \caption{Symmetric chain decomposition for $L(3, 10)$ matching that of Figure \ref{fig:SCD-w3-C0}.
      The blue edges follow rule 1, red edges follow rule 2, magenta edges follow rules 3,4 and 5,
      orange edges follow rules 6,7 and 8.
      \label{figure.3d}}
\end{figure}
\end{example}

\subsubsection{The case $w=4$}
\label{section.w=4}

We give a new symmetric chain decomposition of $L(4,h)$ satisfying the restriction and pattern condition (and hence the extension condition).
As for $w=3$, we are able to describe the minimal and maximal elements explicitly.
As a consequence, we get an explicit combinatorial description for $a^\lambda_{(4),(h)}$ in~\eqref{equation.b lambda} when $w=4$.

\bigskip

\begin{theorem}
\label{theorem.w=4}
    Let $h \in \NN$.
    \begin{enumerate}
        \item
            There exists a symmetric chain decomposition $\mathscr C$ of
            $L(4, h)$ satisfying the restriction and pattern condition.

        \item
            $\lambda$ is a minimal element of a chain in $\mathscr C$ iff
            $\lambda \in L(4, h)$ and $\lambda = 4^{m_4} 2^{m_2} 1^{m_1}$ with
            \begin{equation*}
                {m_2} \in 2\ZZ, \qquad
                {m_1} \geqslant 2{m_4}, \qquad
                {m_1} \neq 2{m_4}+1.
            \end{equation*}

        \item
            If $\lambda$ is a minimal element of the form $\lambda = 4^{m_4} 2^{2m_2} 1^{m_1+2m_4}$ with $m_1\neq1$ in a chain $C$,
            then the corresponding maximal element in $C$ is
            \begin{enumerate}
            \item $4^{h-m_1-2m_2-2m_4} 3^{m_1} 2^{2m_2} 1^{2m_4}$ if $m_1$ is even;
            \item $4^{h-m_1-2m_2-2m_4+1} 3^{m_1-2} 2^{2m_2+1} 1^{2m_4}$ if $m_1$ is odd.
            \end{enumerate}

        \item
            $\lambda$ is a maximal element of a chain in $\mathscr C$ iff
            its complementary shape is of the form $3^{m_3} 2^{m_2} 1^{m_1}$ and satisfies
            \begin{itemize}[noitemsep]
                \item $m_3$ and $m_1 + m_2$ are both even;
                \item ${m_1 + m_2} \leqslant -\frac{3}{2} {m_3} + h$
                    with equality iff $m_1$ and $m_2$ are both even.
            \end{itemize}
    \end{enumerate}
\end{theorem}

\begin{corollary}
\label{cor.b w=4}
We have $s_4[s_h](x,y) = \sum_{\lambda=(\lambda_1,\lambda_2) \vdash 4h} a^\lambda_{(4),(h)} s_\lambda(x,y)$ with
\begin{multline*}
	a^\lambda_{(4),(h)} = \#\{ \mu \vdash \lambda_2 \mid \mu=4^{m_4} 2^{m_2} 1^{m_1} \text{ such that } m_1+m_2+m_4\leqslant h,\\
	 {m_2} \in 2\ZZ, {m_1} \geqslant 2{m_4}, {m_1} \neq 2{m_4}+1\}.
\end{multline*}
\end{corollary}

\begin{figure}[htpb]
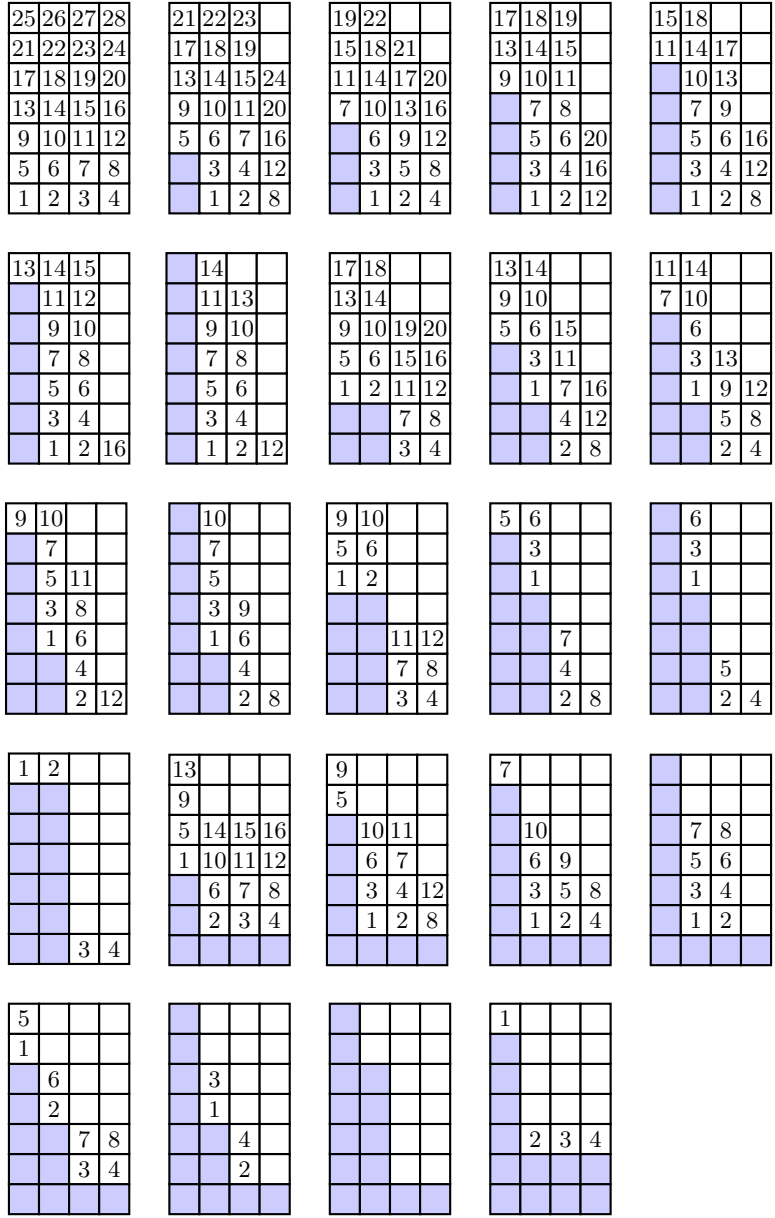

    \centering
    \begin{equation*}

    \end{equation*}
    \caption{The symmetric chain decomposition for $L(4, 7)$ satisfying the restriction,
extension and pattern condition.  The chains are organized so that the elements
of $L''(4,h)$ are all listed first, followed by the elements of
$L'(4,h) \backslash L''(4,h)$, followed by the elements of $L(4,h) \backslash L'(4,h)$.}
    \label{fig:SCD-w4-C0C4C8}
\end{figure}

As in the case of $L(3,h)$, we decompose $L(4,h)$ into smaller components.
Let $L'(4,h)$ be the set of partitions $\lambda= 4^{m_4} 3^{m_3} 2^{m_2} 1^{m_1}$ in $L(4,h)$ such that either $m_4=0$ or $m_1\leqslant 1$.
That is, $L'(4,h)$ is the set of partitions which do not contain $(4,1,1)$ as a subpartition. We have
\begin{equation}
\label{equation.L4 decomp}
	L(4,h) = L'(4,h) \uplus \{ \lambda\oplus(4,1,1) \mid  \lambda \in L(4,h-3)\}.
\end{equation}
Thus, if $L'(4,h)$ has a symmetric chain decomposition, so does $L(4,h)$. We decompose $L'(4,h)$ further.
Let $L''(4,h)$ be the set of partitions $\lambda= 4^{m_4} 3^{m_3} 2^{m_2} 1^{m_1}$ in $L'(4,h)$ such that $m_2 \leqslant 1$.
That is, $L''(4,h)$ is the set of partitions which contain neither $(4,1,1)$ nor $(2,2)$ as a subpartition. We have
\begin{equation}
\label{equation.L4 decomp a}
	L'(4,h) = L''(4,h) \uplus \{ \lambda\oplus(2,2) \mid  \lambda \in L'(4,h-2)\}.
\end{equation}
As in the $w=3$ case, we can proceed by induction since the chains of
$L'(4,h)$, $\{ \lambda\oplus(4,1,1) \mid \lambda \in L(4,h-3)\}$, and  $\{ \lambda\oplus(2,2) \mid  \lambda \in L'(4,h-2)\}$
are all centered around $2h = \frac{4h}{2} = \frac{4(h-3)}{2} +6 = \frac{4(h-2)}{2}+4$.
Hence the proof of Theorem~\ref{theorem.w=4} follows from the following proposition.

\begin{proposition}
There is a symmetric chain decomposition $\mathscr C$ of $L''(4,h)$ satisfying:
\begin{enumerate}
\item $\mathscr C$ satisfies the restriction and pattern condition.
\item The minimal element of a chain $C$ in $\mathscr C$ is of the form $1^{m_1}$ with $m_1 \neq 1$.
\item  If the minimal element of a chain $C$ in $\mathscr C$ is $1^{m_1}$, then the corresponding maximal element in $C$ is
\begin{enumerate}
\item $4^{h-m_1}3^{m_1}$ if $m_1$ is even;
\item $4^{h-m_1+1}3^{m_1-2}2$ if $m_1$ is odd.
\end{enumerate}
\end{enumerate}
\end{proposition}

\begin{proof}
The proof proceeds in the same fashion as the proof of Proposition~\ref{proposition.L'}. Again, we denote by $\mathsf{HW}$ the minimal
(or highest weight) elements and by $\mathsf{LW}$ the maximal (or lowest weight) elements in the chains.

We define two functions,
\[
	f \colon L''(4,h) \backslash \mathsf{LW} \rightarrow \{1,2,3,4\} \quad \text{and} \quad
	e \colon L''(4,h) \backslash \mathsf{HW} \rightarrow \{1,2,3,4\}
\]
which, if $\lambda^{(c)} \lessdot \lambda^{(c+1)}$ in the symmetric chain decomposition,
then $f(\lambda^{(c)})= e(\lambda^{(c+1)})$ is equal to the column index of the cell $\lambda^{(c+1)}\backslash\lambda^{(c)}$.

Let $\lambda = 4^{m_4} 3^{m_3} 2^{m_2} 1^{m_1} \in L''(4,h)$ and set $d = m_1+m_2+m_3$.
We say that $\lambda$ is in \defn{phase~1} if $m_4=0$, $d$ is even and
$m_1+m_2 \geqslant 1$, or $d$ is odd and $m_1 > 1$.
Otherwise $\lambda$ is in \defn{phase~2}.
All highest weight elements except the empty partition are in phase 1.
If the highest weight element is $1^{m_1}$ with $m_1 \neq 1$,
the corresponding lowest weight is $4^{h-m_1}3^{m_1}$ if $m_1$ is even
and $4^{h-m_1+1}3^{m_1-2}2$ if $m_1$ is odd.

Every partition $\lambda = 4^{m_4}3^{m_3}2^{m_2}1^{m_1}
\in L''(4,h)$ is in one of the following 10 cases.
The partition $\lambda$ is in phase 1 and
\begin{enumerate}
\item $\lambda$ has even length. In this case,
\begin{enumerate}
\item if $m_2=0$, then $f(\lambda) = 2$ and $e(\lambda)=3$ (unless $m_3=0$, in which case $\lambda \in \mathsf{HW}$).
\item if $m_2=1$, then $f(\lambda) = 3$ and $e(\lambda) = 2$.
If $m_1=0$ then the next partition is in phase 2.
\end{enumerate}
The highest weight element is $1^{\ell(\lambda)}$ and the lowest weight element is $4^{h-\ell(\lambda)} 3^{\ell(\lambda)}$.
\item $\lambda$ has odd length. In this case,
\begin{enumerate}
\item if $m_2=0$, then $f(\lambda)=2$ and $e(\lambda)=3$ (unless $m_3=0$, in which case $\lambda\in \mathsf{HW}$).
If $m_1=2$ then the next partition is in phase 2.
\item if $m_2=1$, then $f(\lambda)=3$ and $e(\lambda) = 2$.
\end{enumerate}
The highest weight element is $1^{\ell(\lambda)}$ and the lowest weight element is $4^{h-\ell(\lambda)+1} 3^{\ell(\lambda)-2}2$.
\end{enumerate}
Otherwise $\lambda$ is in phase 2 and
\begin{enumerate}
\item[(3)] $m_1=0$, $m_2=0$ and $m_3$ is even. Then
$f(\lambda)=1$ (unless $\ell(\lambda)=h$ in which case $\lambda \in \mathsf{LW}$) and $e(\lambda)=4$ if $m_4>0$
(if $m_4=0$ then $e(\lambda)=3$ and the previous element is in phase 1).
The highest weight of this chain is $1^{m_3}$ and the corresponding lowest weight element is $4^{h-m_3}3^{m_3}$.
\item[(4)] $m_1=1$, $m_2=0$ and $m_3$ is even. Then
$f(\lambda)=2$ and $e(\lambda)=1$.
The highest weight of this chain is $1^{m_3}$ and the corresponding lowest weight element is $4^{h-m_3}3^{m_3}$.
\item[(5)] $m_1=0$, $m_2=1$ and $m_3$ is even. Then
$f(\lambda)=3$ and $e(\lambda)=2$.
The highest weight of this chain is $1^{m_3}$ and the corresponding lowest weight element is $4^{h-m_3}3^{m_3}$.
\item[(6)] $m_1=0$, $m_2=0$ and $m_3$ is odd. Then
$f(\lambda)=4$ and $e(\lambda)=3$.
The highest weight of this chain is $1^{m_3-1}$ and the corresponding lowest weight element is $4^{h-m_3+1}3^{m_3-1}$.
\item[(7)] $m_1=1$, $m_2=1$ and $m_3$ is odd. Then
$f(\lambda)=4$.  In this case, $e(\lambda)=1$ if $m_4>0$ (if $m_4=0$ then $e(\lambda)=2$ and the previous element
is in phase 1).
The highest weight element of this chain is $1^{m_3+2}$ and the corresponding lowest weight element is
$4^{h-m_3-1} 3^{m_3} 2$.
\item[(8)] $m_1=1$, $m_2=1$ and $m_3$ is even. Then
$f(\lambda)=3$ and $e(\lambda)=4$.
The highest weight element of this chain is $1^{m_3+3}$ and the corresponding lowest weight element is
$4^{h-m_3-2} 3^{m_3+1} 2$.
\item[(9)] $m_1=1$, $m_2=0$ and $m_3$ is odd. Then
$f(\lambda)=2$ and $e(\lambda)=3$.
The highest weight element of this chain is $1^{m_3+2}$ and the corresponding lowest weight element is
$4^{h-m_3-1} 3^{m_3} 2$.
\item[(10)] $m_1=0$, $m_2=1$ and $m_3$ is odd. If $m_2+m_3+m_4= h$, then $\lambda$ is of $\mathsf{LW}$.
Otherwise, $f(\lambda)=1$ and $e(\lambda)=2$.
The highest weight element of this chain is $1^{m_3+2}$ and the corresponding lowest weight element is
$4^{h-m_3-1} 3^{m_3} 2$.
\end{enumerate}
One can check explicitly, that $f(e(\lambda))=\lambda$ if $e(\lambda)$ is defined and $e(f(\lambda))=\lambda$ if $f(\lambda)$ is defined. Hence they
are partial inverses of each other and indeed define a symmetric chain decomposition.
If the highest weight element is $1^{m_1}$ with $m_1$ even (resp. odd), then eventually the chain defined by the operator $f$ will oscillate between
cases $(3)\to (4) \to (5) \to (6) \to (3) \to \cdots$ (resp. $(7) \to (8) \to (9) \to (10) \to (7) \to \cdots$) with pattern 1234 (resp. 4321). Hence the chains satisfy the
pattern condition. The restriction condition follows by construction.
\end{proof}

\begin{example}
See Figure~\ref{fig:SCD-w4-C0C4C8} for the symmetric chain decomposition for $L(4,7)$.
\end{example}

The decomposition of $L(4,h)$ appearing in Equations~\eqref{equation.L4 decomp} and~\eqref{equation.L4 decomp a}
implies a new recursive formula (see Corollary~\ref{corollary.L4} below) analogous to a formula by T\'etreault~\cite{Tetreault2020} which 
is related to the decomposition of $L(3,h)$. T\'etreault's formula~\cite[Theorem~2.1]{Tetreault2020} states that for all $h \geqslant 1$
\[
	s_3[s_h]\!\!\downarrow_2
	= s_{(6,6)} \odot (s_3[s_{h-4}]\!\!\downarrow_2)
	+ s_{3h} + \sum_{k=2}^h s_{(3h-k,k)}~.
\]
Here $s_\mu \!\!\downarrow_k=s_\mu$ if $\ell(\mu)\leqslant k$ and 0 otherwise. Furthermore, $s_\lambda \odot s_\mu = s_{\lambda+\mu}$,
where the sum of two partitions is done componentwise adjoining parts of size zero when necessary.
This recursive formula follows from Equation~\eqref{equation.L decomp} and Proposition~\ref{proposition.L'} since
\[
\begin{split}
	s_3[s_h]\!\!\downarrow_2 &= \sum_{\lambda \in L(3,h)} s_{(3k-|\lambda|,|\lambda|)},\\
	s_{(6,6)} \odot (s_3[s_{h-4}]\!\!\downarrow_2) &= \sum_{\lambda \in L'(3,h)} s_{(3k-|\lambda|,|\lambda|)} \hbox{ and }\\
	s_{3h} + \sum_{k=2}^h s_{(3h-k,k)} &= \sum_{\lambda \in L(3,h) \backslash L'(3,h)} s_{(3k-|\lambda|,|\lambda|)}~.
\end{split}
\]

T\'etreault states in the last line of \cite{Tetreault2020} that
``$\ldots$ even in the case $h_4[h_n]$, such a recurrence
is hard to find.''
Here we observe that the decompositions~\eqref{equation.L4 decomp} and~\eqref{equation.L4 decomp a} for $L(4,h)$
imply such a recursion stated in the following corollary.

\begin{corollary}
\label{corollary.L4}
For a symmetric function
$f = \sum_{\lambda} c_\lambda s_\lambda$, define
$f\!\!\downarrow_2 = \sum_{\ell(\lambda) \leqslant 2} c_\lambda s_\lambda$.
Then for all $h\geqslant 1$
\begin{multline*}
	s_4[s_h]\!\!\downarrow_2 = s_{(6,6)} \odot (s_4[s_{h-3}]\!\!\downarrow_2)\\
	+ s_{(4,4)} \odot (s_4[s_{h-2}]\!\!\downarrow_2)
	- s_{(10,10)} \odot (s_4[s_{h-5}]\!\!\downarrow_2)
	+ s_{4h} + \sum_{k=2}^h s_{(4h-k,k)}~.
\end{multline*}
\end{corollary}

\begin{proof}
We say that a symmetric function $f$ enumerates a set of partitions $A \subseteq L(4,h)$ if
$f = \sum_{\lambda \in A} s_{(4h - |\lambda|,|\lambda|)}$.
For example, $s_4[s_{h}]\!\!\downarrow_2$ enumerates the set of partitions $L(4,h)$.

The partitions in $L(4,h)$ may be divided into three sets (as in the example in Figure \ref{fig:SCD-w4-C0C4C8}):
\begin{itemize}
\item Those in $L(4,h) \backslash L'(4,h)$ which contain at least one copy of
$(4,1,1)$.  These partitions are enumerated by
$s_{(6,6)} \odot s_4[s_{h-3}]\!\!\downarrow_2$.
\item Those in $L'(4,h) \backslash L''(4,h)$ which contain at least one copy of
$(2,2)$ as a subpartition, but no copies of $(4,1,1)$.  These partitions are enumerated
by $s_{(4,4)} \odot (s_4[s_{h-2}]\!\!\downarrow_2 - s_{(6,6)} \odot s_4[s_{h-5}]\!\!\downarrow_2)$
\item Those in $L''(4,h)$ which contain neither $(4,1,1)$ nor $(2,2)$ as a subpartition.
These partitions are enumerated by $s_{4h} + \sum_{k=2}^h s_{(4h-k,k)}$.
\end{itemize}
The sum of these three expressions is the right hand side of the expression stated in the
corollary.
\end{proof}

\bibliographystyle{plain}
\bibliography{quasi_symmetric}

\newpage

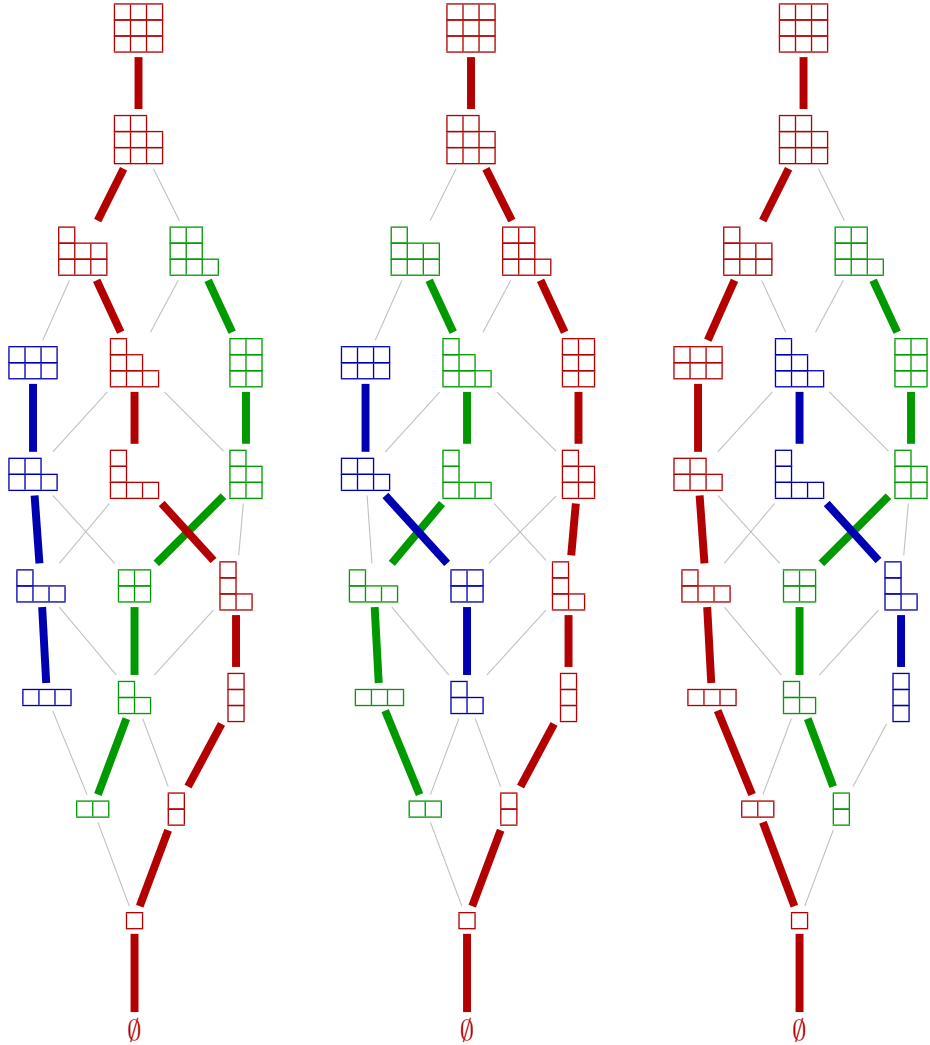
\begin{figure}[htpb]
    \centering
    \Yboxdim{0.6em}
    \Ylinethick{0.5pt}
    \begin{tikzpicture}[>=latex,line join=bevel, yscale=0.8, xscale=0.75, every node/.style={inner sep=2pt}]
      \node (node_0) at (65.39bp,7.3332bp) [draw,draw=none] {${\color{darkred}\emptyset}$};
      \node (node_1) at (65.39bp,58.968bp) [draw,draw=none] {${\color{darkred}\yng(1)}$};
      \node (node_2) at (44.39bp,111.57bp) [draw,draw=none] {${\color{darkgreen}\yng(2)}$};
      \node (node_4) at (86.39bp,111.57bp) [draw,draw=none] {${\color{darkred}\yng(1,1)}$};
      \node (node_3) at (21.39bp,164.18bp) [draw,draw=none] {${\color{darkblue}\yng(3)}$};
      \node (node_5) at (65.39bp,164.18bp) [draw,draw=none] {${\color{darkgreen}\yng(1,2)}$};
      \node (node_6) at (18.39bp,216.78bp) [draw,draw=none] {${\color{darkblue}\yng(1,3)}$};
      \node (node_10) at (116.39bp,164.18bp) [draw,draw=none] {${\color{darkred}\yng(1,1,1)}$};
      \node (node_7) at (65.39bp,216.78bp) [draw,draw=none] {${\color{darkgreen}\yng(2,2)}$};
      \node (node_11) at (116.39bp,216.78bp) [draw,draw=none] {${\color{darkred}\yng(1,1,2)}$};
      \node (node_8) at (14.39bp,269.38bp) [draw,draw=none] {${\color{darkblue}\yng(2,3)}$};
      \node (node_12) at (65.39bp,269.38bp) [draw,draw=none] {${\color{darkred}\yng(1,1,3)}$};
      \node (node_13) at (121.39bp,269.38bp) [draw,draw=none] {${\color{darkgreen}\yng(1,2,2)}$};
      \node (node_9) at (14.39bp,321.99bp) [draw,draw=none] {${\color{darkblue}\yng(3,3)}$};
      \node (node_14) at (65.39bp,321.99bp) [draw,draw=none] {${\color{darkred}\yng(1,2,3)}$};
      \node (node_15) at (39.39bp,374.59bp) [draw,draw=none] {${\color{darkred}\yng(1,3,3)}$};
      \node (node_16) at (121.39bp,321.99bp) [draw,draw=none] {${\color{darkgreen}\yng(2,2,2)}$};
      \node (node_17) at (95.39bp,374.59bp) [draw,draw=none] {${\color{darkgreen}\yng(2,2,3)}$};
      \node (node_18) at (67.39bp,427.19bp) [draw,draw=none] {${\color{darkred}\yng(2,3,3)}$};
      \node (node_19) at (67.39bp,479.8bp) [draw,draw=none] {${\color{darkred}\yng(3,3,3)}$};
      \draw[darkred,line width=3] (node_0) to (node_1);
      \draw [black,color=Gray!50] (node_1) to (node_2);
      \draw[darkred,line width=3] (node_1) to (node_4);
      \draw [black,color=Gray!50] (node_2) to (node_3);
      \draw[darkgreen,line width=3] (node_2) to (node_5);
      \draw[darkblue,line width=3] (node_3) to (node_6);
      \draw [black,color=Gray!50] (node_4) to (node_5);
      \draw[darkred,line width=3] (node_4) to (node_10);
      \draw [black,color=Gray!50] (node_5) to (node_6);
      \draw[darkgreen,line width=3] (node_5) to (node_7);
      \draw [black,color=Gray!50] (node_5) to (node_11);
      \draw[darkblue,line width=3] (node_6) to (node_8);
      \draw [black,color=Gray!50] (node_6) to (node_12);
      \draw [black,color=Gray!50] (node_7) to (node_8);
      \draw[darkgreen,line width=3] (node_7) to (node_13);
      \draw[darkblue,line width=3] (node_8) to (node_9);
      \draw [black,color=Gray!50] (node_8) to (node_14);
      \draw [black,color=Gray!50] (node_9) to (node_15);
      \draw[darkred,line width=3] (node_10) to (node_11);
      \draw[darkred,line width=3] (node_11) to (node_12);
      \draw [black,color=Gray!50] (node_11) to (node_13);
      \draw[darkred,line width=3] (node_12) to (node_14);
      \draw [black,color=Gray!50] (node_13) to (node_14);
      \draw[darkgreen,line width=3] (node_13) to (node_16);
      \draw[darkred,line width=3] (node_14) to (node_15);
      \draw [black,color=Gray!50] (node_14) to (node_17);
      \draw[darkred,line width=3] (node_15) to (node_18);
      \draw[darkgreen,line width=3] (node_16) to (node_17);
      \draw [black,color=Gray!50] (node_17) to (node_18);
      \draw[darkred,line width=3] (node_18) to (node_19);
    \end{tikzpicture}
    \qquad
    \begin{tikzpicture}[>=latex,line join=bevel, yscale=0.8, xscale=0.75, every node/.style={inner sep=2pt}]
      \node (node_0) at (65.39bp,7.3332bp) [draw,draw=none] {${\color{darkred}\emptyset}$};
      \node (node_1) at (65.39bp,58.968bp) [draw,draw=none] {${\color{darkred}\yng(1)}$};
      \node (node_2) at (44.39bp,111.57bp) [draw,draw=none] {${\color{darkgreen}\yng(2)}$};
      \node (node_4) at (86.39bp,111.57bp) [draw,draw=none] {${\color{darkred}\yng(1,1)}$};
      \node (node_3) at (21.39bp,164.18bp) [draw,draw=none] {${\color{darkgreen}\yng(3)}$};
      \node (node_5) at (65.39bp,164.18bp) [draw,draw=none] {${\color{darkblue}\yng(1,2)}$};
      \node (node_6) at (18.39bp,216.78bp) [draw,draw=none] {${\color{darkgreen}\yng(1,3)}$};
      \node (node_10) at (116.39bp,164.18bp) [draw,draw=none] {${\color{darkred}\yng(1,1,1)}$};
      \node (node_7) at (65.39bp,216.78bp) [draw,draw=none] {${\color{darkblue}\yng(2,2)}$};
      \node (node_11) at (116.39bp,216.78bp) [draw,draw=none] {${\color{darkred}\yng(1,1,2)}$};
      \node (node_8) at (14.39bp,269.38bp) [draw,draw=none] {${\color{darkblue}\yng(2,3)}$};
      \node (node_12) at (65.39bp,269.38bp) [draw,draw=none] {${\color{darkgreen}\yng(1,1,3)}$};
      \node (node_13) at (121.39bp,269.38bp) [draw,draw=none] {${\color{darkred}\yng(1,2,2)}$};
      \node (node_9) at (14.39bp,321.99bp) [draw,draw=none] {${\color{darkblue}\yng(3,3)}$};
      \node (node_14) at (65.39bp,321.99bp) [draw,draw=none] {${\color{darkgreen}\yng(1,2,3)}$};
      \node (node_15) at (39.39bp,374.59bp) [draw,draw=none] {${\color{darkgreen}\yng(1,3,3)}$};
      \node (node_16) at (121.39bp,321.99bp) [draw,draw=none] {${\color{darkred}\yng(2,2,2)}$};
      \node (node_17) at (95.39bp,374.59bp) [draw,draw=none] {${\color{darkred}\yng(2,2,3)}$};
      \node (node_18) at (67.39bp,427.19bp) [draw,draw=none] {${\color{darkred}\yng(2,3,3)}$};
      \node (node_19) at (67.39bp,479.8bp) [draw,draw=none] {${\color{darkred}\yng(3,3,3)}$};
      \draw[darkred,line width=3] (node_0) to (node_1);
      \draw [black,color=Gray!50] (node_1) to (node_2);
      \draw[darkred,line width=3] (node_1) to (node_4);
      \draw[darkgreen,line width=3] (node_2) to (node_3);
      \draw [black,color=Gray!50] (node_2) to (node_5);
      \draw[darkgreen,line width=3] (node_3) to (node_6);
      \draw [black,color=Gray!50] (node_4) to (node_5);
      \draw[darkred,line width=3] (node_4) to (node_10);
      \draw [black,color=Gray!50] (node_5) to (node_6);
      \draw[darkblue,line width=3] (node_5) to (node_7);
      \draw [black,color=Gray!50] (node_5) to (node_11);
      \draw [black,color=Gray!50] (node_6) to (node_8);
      \draw[darkgreen,line width=3] (node_6) to (node_12);
      \draw[darkblue,line width=3] (node_7) to (node_8);
      \draw [black,color=Gray!50] (node_7) to (node_13);
      \draw[darkblue,line width=3] (node_8) to (node_9);
      \draw [black,color=Gray!50] (node_8) to (node_14);
      \draw [black,color=Gray!50] (node_9) to (node_15);
      \draw[darkred,line width=3] (node_10) to (node_11);
      \draw [black,color=Gray!50] (node_11) to (node_12);
      \draw[darkred,line width=3] (node_11) to (node_13);
      \draw[darkgreen,line width=3] (node_12) to (node_14);
      \draw [black,color=Gray!50] (node_13) to (node_14);
      \draw[darkred,line width=3] (node_13) to (node_16);
      \draw[darkgreen,line width=3] (node_14) to (node_15);
      \draw [black,color=Gray!50] (node_14) to (node_17);
      \draw [black,color=Gray!50] (node_15) to (node_18);
      \draw[darkred,line width=3] (node_16) to (node_17);
      \draw[darkred,line width=3] (node_17) to (node_18);
      \draw[darkred,line width=3] (node_18) to (node_19);
    \end{tikzpicture}
    \qquad
    \begin{tikzpicture}[>=latex,line join=bevel, yscale=0.8, xscale=0.75, every node/.style={inner sep=2pt}]
      \node (node_0) at (65.39bp,7.3332bp) [draw,draw=none] {${\color{darkred}\emptyset}$};
      \node (node_1) at (65.39bp,58.968bp) [draw,draw=none] {${\color{darkred}\yng(1)}$};
      \node (node_2) at (44.39bp,111.57bp) [draw,draw=none] {${\color{darkred}\yng(2)}$};
      \node (node_4) at (86.39bp,111.57bp) [draw,draw=none] {${\color{darkgreen}\yng(1,1)}$};
      \node (node_3) at (21.39bp,164.18bp) [draw,draw=none] {${\color{darkred}\yng(3)}$};
      \node (node_5) at (65.39bp,164.18bp) [draw,draw=none] {${\color{darkgreen}\yng(1,2)}$};
      \node (node_6) at (18.39bp,216.78bp) [draw,draw=none] {${\color{darkred}\yng(1,3)}$};
      \node (node_10) at (116.39bp,164.18bp) [draw,draw=none] {${\color{darkblue}\yng(1,1,1)}$};
      \node (node_7) at (65.39bp,216.78bp) [draw,draw=none] {${\color{darkgreen}\yng(2,2)}$};
      \node (node_11) at (116.39bp,216.78bp) [draw,draw=none] {${\color{darkblue}\yng(1,1,2)}$};
      \node (node_8) at (14.39bp,269.38bp) [draw,draw=none] {${\color{darkred}\yng(2,3)}$};
      \node (node_12) at (65.39bp,269.38bp) [draw,draw=none] {${\color{darkblue}\yng(1,1,3)}$};
      \node (node_13) at (121.39bp,269.38bp) [draw,draw=none] {${\color{darkgreen}\yng(1,2,2)}$};
      \node (node_9) at (14.39bp,321.99bp) [draw,draw=none] {${\color{darkred}\yng(3,3)}$};
      \node (node_14) at (65.39bp,321.99bp) [draw,draw=none] {${\color{darkblue}\yng(1,2,3)}$};
      \node (node_15) at (39.39bp,374.59bp) [draw,draw=none] {${\color{darkred}\yng(1,3,3)}$};
      \node (node_16) at (121.39bp,321.99bp) [draw,draw=none] {${\color{darkgreen}\yng(2,2,2)}$};
      \node (node_17) at (95.39bp,374.59bp) [draw,draw=none] {${\color{darkgreen}\yng(2,2,3)}$};
      \node (node_18) at (67.39bp,427.19bp) [draw,draw=none] {${\color{darkred}\yng(2,3,3)}$};
      \node (node_19) at (67.39bp,479.8bp) [draw,draw=none] {${\color{darkred}\yng(3,3,3)}$};
      \draw[darkred,line width=3] (node_0) to (node_1);
      \draw[darkred,line width=3] (node_1) to (node_2);
      \draw [black,color=Gray!50] (node_1) to (node_4);
      \draw[darkred,line width=3] (node_2) to (node_3);
      \draw [black,color=Gray!50] (node_2) to (node_5);
      \draw[darkred,line width=3] (node_3) to (node_6);
      \draw[darkgreen,line width=3] (node_4) to (node_5);
      \draw [black,color=Gray!50] (node_4) to (node_10);
      \draw [black,color=Gray!50] (node_5) to (node_6);
      \draw[darkgreen,line width=3] (node_5) to (node_7);
      \draw [black,color=Gray!50] (node_5) to (node_11);
      \draw[darkred,line width=3] (node_6) to (node_8);
      \draw [black,color=Gray!50] (node_6) to (node_12);
      \draw [black,color=Gray!50] (node_7) to (node_8);
      \draw[darkgreen,line width=3] (node_7) to (node_13);
      \draw[darkred,line width=3] (node_8) to (node_9);
      \draw [black,color=Gray!50] (node_8) to (node_14);
      \draw[darkred,line width=3] (node_9) to (node_15);
      \draw[darkblue,line width=3] (node_10) to (node_11);
      \draw[darkblue,line width=3] (node_11) to (node_12);
      \draw [black,color=Gray!50] (node_11) to (node_13);
      \draw[darkblue,line width=3] (node_12) to (node_14);
      \draw [black,color=Gray!50] (node_13) to (node_14);
      \draw[darkgreen,line width=3] (node_13) to (node_16);
      \draw [black,color=Gray!50] (node_14) to (node_15);
      \draw [black,color=Gray!50] (node_14) to (node_17);
      \draw[darkred,line width=3] (node_15) to (node_18);
      \draw[darkgreen,line width=3] (node_16) to (node_17);
      \draw [black,color=Gray!50] (node_17) to (node_18);
      \draw[darkred,line width=3] (node_18) to (node_19);
    \end{tikzpicture}
    \caption{Lindstrom's~\cite{Lindstrom1980}, Wen's~\cite{Wen2004}, and our SCD of $L(3, 3)$}
    \label{fig:SCDsL33}
\end{figure}

\begin{figure}[htpb]
    \centering
    \Yboxdim{0.6em}
    \Ylinethick{0.5pt}
    \begin{tikzpicture}[>=latex,line join=bevel, yscale=0.8, every node/.style={inner sep=3pt}]
      \node (node_0) at (106.39bp,7.3332bp) [draw,draw=none] {${\color{darkred}\emptyset}$};
      \node (node_1) at (106.39bp,58.968bp) [draw,draw=none] {${\color{darkred}\yng(1)}$};
      \node (node_2) at (85.39bp,111.57bp) [draw,draw=none] {${\color{darkred}\yng(2)}$};
      \node (node_5) at (127.39bp,111.57bp) [draw,draw=none] {${\color{RoyalBlue}\yng(1,1)}$};
      \node (node_3) at (59.39bp,164.18bp) [draw,draw=none] {${\color{darkred}\yng(3)}$};
      \node (node_6) at (106.39bp,164.18bp) [draw,draw=none] {${\color{RoyalBlue}\yng(1,2)}$};
      \node (node_4) at (16.39bp,216.78bp) [draw,draw=none] {${\color{darkred}\yng(4)}$};
      \node (node_7) at (59.39bp,216.78bp) [draw,draw=none] {${\color{RoyalBlue}\yng(1,3)}$};
      \node (node_8) at (14.39bp,269.38bp) [draw,draw=none] {${\color{darkred}\yng(1,4)}$};
      \node (node_15) at (157.39bp,164.18bp) [draw,draw=none] {${\color{ForestGreen}\yng(1,1,1)}$};
      \node (node_9) at (106.39bp,216.78bp) [draw,draw=none] {${\color{YellowOrange}\yng(2,2)}$};
      \node (node_16) at (157.39bp,216.78bp) [draw,draw=none] {${\color{ForestGreen}\yng(1,1,2)}$};
      \node (node_10) at (61.39bp,269.38bp) [draw,draw=none] {${\color{RoyalBlue}\yng(2,3)}$};
      \node (node_17) at (112.39bp,269.38bp) [draw,draw=none] {${\color{ForestGreen}\yng(1,1,3)}$};
      \node (node_11) at (14.39bp,321.99bp) [draw,draw=none] {${\color{darkred}\yng(2,4)}$};
      \node (node_18) at (112.39bp,321.99bp) [draw,draw=none] {${\color{Cerulean}\yng(1,1,4)}$};
      \node (node_19) at (168.39bp,269.38bp) [draw,draw=none] {${\color{YellowOrange}\yng(1,2,2)}$};
      \node (node_12) at (61.39bp,321.99bp) [draw,draw=none] {${\color{RoyalBlue}\yng(3,3)}$};
      \node (node_20) at (168.39bp,321.99bp) [draw,draw=none] {${\color{ForestGreen}\yng(1,2,3)}$};
      \node (node_13) at (61.39bp,374.59bp) [draw,draw=none] {${\color{darkred}\yng(3,4)}$};
      \node (node_21) at (112.39bp,374.59bp) [draw,draw=none] {${\color{ForestGreen}\yng(1,2,4)}$};
      \node (node_22) at (168.39bp,374.59bp) [draw,draw=none] {${\color{RoyalBlue}\yng(1,3,3)}$};
      \node (node_14) at (61.39bp,427.19bp) [draw,draw=none] {${\color{darkred}\yng(4,4)}$};
      \node (node_23) at (112.39bp,427.19bp) [draw,draw=none] {${\color{ForestGreen}\yng(1,3,4)}$};
      \node (node_24) at (112.39bp,479.8bp) [draw,draw=none] {${\color{darkred}\yng(1,4,4)}$};
      \node (node_25) at (224.39bp,321.99bp) [draw,draw=none] {${\color{YellowOrange}\yng(2,2,2)}$};
      \node (node_26) at (224.39bp,374.59bp) [draw,draw=none] {${\color{YellowOrange}\yng(2,2,3)}$};
      \node (node_27) at (168.39bp,427.19bp) [draw,draw=none] {${\color{YellowOrange}\yng(2,2,4)}$};
      \node (node_28) at (224.39bp,427.19bp) [draw,draw=none] {${\color{RoyalBlue}\yng(2,3,3)}$};
      \node (node_29) at (168.39bp,479.8bp) [draw,draw=none] {${\color{ForestGreen}\yng(2,3,4)}$};
      \node (node_30) at (140.39bp,532.4bp) [draw,draw=none] {${\color{darkred}\yng(2,4,4)}$};
      \node (node_31) at (224.39bp,479.8bp) [draw,draw=none] {${\color{RoyalBlue}\yng(3,3,3)}$};
      \node (node_32) at (196.39bp,532.4bp) [draw,draw=none] {${\color{RoyalBlue}\yng(3,3,4)}$};
      \node (node_33) at (168.39bp,585.01bp) [draw,draw=none] {${\color{darkred}\yng(3,4,4)}$};
      \node (node_34) at (168.39bp,637.61bp) [draw,draw=none] {${\color{darkred}\yng(4,4,4)}$};
      \draw [darkred,line width=2.5] (node_0) to (node_1);
      \draw [darkred,line width=2.5] (node_1) to (node_2);
      \draw [black,color=Gray!50] (node_1) to (node_5);
      \draw [darkred,line width=2.5] (node_2) to (node_3);
      \draw [black,color=Gray!50] (node_2) to (node_6);
      \draw [darkred,line width=2.5] (node_3) to (node_4);
      \draw [black,color=Gray!50] (node_3) to (node_7);
      \draw [darkred,line width=2.5] (node_4) to (node_8);
      \draw [RoyalBlue,line width=2.5] (node_5) to (node_6);
      \draw [black,color=Gray!50] (node_5) to (node_15);
      \draw [RoyalBlue,line width=2.5] (node_6) to (node_7);
      \draw [black,color=Gray!50] (node_6) to (node_9);
      \draw [black,color=Gray!50] (node_6) to (node_16);
      \draw [black,color=Gray!50] (node_7) to (node_8);
      \draw [RoyalBlue,line width=2.5] (node_7) to (node_10);
      \draw [black,color=Gray!50] (node_7) to (node_17);
      \draw [darkred,line width=2.5] (node_8) to (node_11);
      \draw [black,color=Gray!50] (node_8) to (node_18);
      \draw [black,color=Gray!50] (node_9) to (node_10);
      \draw [YellowOrange,line width=2.5] (node_9) to (node_19);
      \draw [black,color=Gray!50] (node_10) to (node_11);
      \draw [RoyalBlue,line width=2.5] (node_10) to (node_12);
      \draw [black,color=Gray!50] (node_10) to (node_20);
      \draw [darkred,line width=2.5] (node_11) to (node_13);
      \draw [black,color=Gray!50] (node_11) to (node_21);
      \draw [black,color=Gray!50] (node_12) to (node_13);
      \draw [RoyalBlue,line width=2.5] (node_12) to (node_22);
      \draw [darkred,line width=2.5] (node_13) to (node_14);
      \draw [black,color=Gray!50] (node_13) to (node_23);
      \draw [darkred,line width=2.5] (node_14) to (node_24);
      \draw [ForestGreen,line width=2.5] (node_15) to (node_16);
      \draw [ForestGreen,line width=2.5] (node_16) to (node_17);
      \draw [black,color=Gray!50] (node_16) to (node_19);
      \draw [black,color=Gray!50] (node_17) to (node_18);
      \draw [ForestGreen,line width=2.5] (node_17) to (node_20);
      \draw [black,color=Gray!50] (node_18) to (node_21);
      \draw [black,color=Gray!50] (node_19) to (node_20);
      \draw [YellowOrange,line width=2.5] (node_19) to (node_25);
      \draw [ForestGreen,line width=2.5] (node_20) to (node_21);
      \draw [black,color=Gray!50] (node_20) to (node_22);
      \draw [black,color=Gray!50] (node_20) to (node_26);
      \draw [ForestGreen,line width=2.5] (node_21) to (node_23);
      \draw [black,color=Gray!50] (node_21) to (node_27);
      \draw [black,color=Gray!50] (node_22) to (node_23);
      \draw [RoyalBlue,line width=2.5] (node_22) to (node_28);
      \draw [black,color=Gray!50] (node_23) to (node_24);
      \draw [ForestGreen,line width=2.5] (node_23) to (node_29);
      \draw [darkred,line width=2.5] (node_24) to (node_30);
      \draw [YellowOrange,line width=2.5] (node_25) to (node_26);
      \draw [YellowOrange,line width=2.5] (node_26) to (node_27);
      \draw [black,color=Gray!50] (node_26) to (node_28);
      \draw [black,color=Gray!50] (node_27) to (node_29);
      \draw [black,color=Gray!50] (node_28) to (node_29);
      \draw [RoyalBlue,line width=2.5] (node_28) to (node_31);
      \draw [black,color=Gray!50] (node_29) to (node_30);
      \draw [black,color=Gray!50] (node_29) to (node_32);
      \draw [darkred,line width=2.5] (node_30) to (node_33);
      \draw [RoyalBlue,line width=2.5] (node_31) to (node_32);
      \draw [black,color=Gray!50] (node_32) to (node_33);
      \draw [darkred,line width=2.5] (node_33) to (node_34);
    \end{tikzpicture}
    \caption{Our SCD of $L(4, 3)$}
    \label{fig:OurSCDL43}
\end{figure}
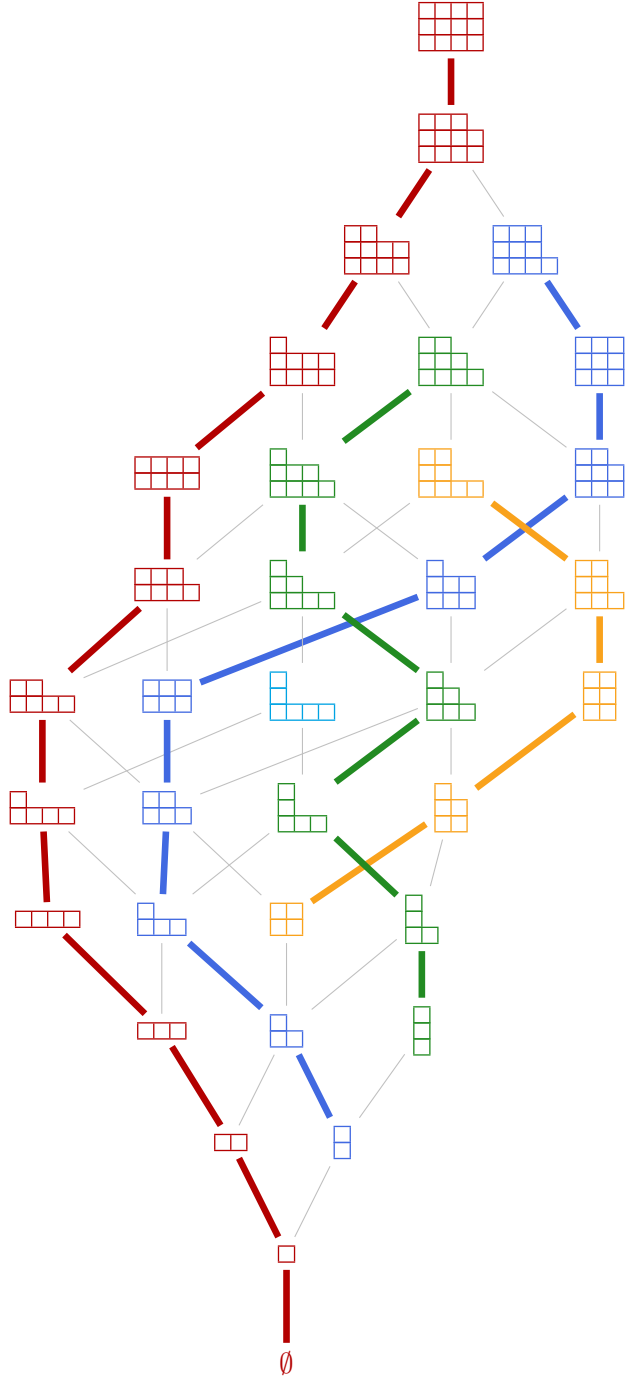

\end{document}